\documentclass[12pt]{amsart}

\usepackage{amsmath, amssymb, amsthm,graphicx}
\usepackage{hyperref}
\usepackage{xcolor}
\usepackage{url}
\usepackage{caption}
\usepackage{algorithm}
\usepackage{algpseudocode}
\usepackage{algorithmicx}
\usepackage{subcaption}
\usepackage{wrapfig}

\theoremstyle{definition}
\newtheorem{lemma}{Lemma}[section]
\newtheorem{theorem}[lemma]{Theorem}

\newtheorem{proposition}[lemma]{Proposition}
\newtheorem{definition}[lemma]{Definition}

\newtheorem{example*}[lemma]{Example}
\newtheorem{remark}[lemma]{Remark}

\theoremstyle{remark}

\DeclareRobustCommand{\qedify}[1]{%
  \ifmmode \quad\hbox{#1}
  \else
    \leavevmode\unskip\penalty9999 \hbox{}\nobreak\hfill
    \quad\hbox{#1}%
  \fi
}

\numberwithin{equation}{section}

\DeclareMathOperator{\supp}{supp}

\DeclareMathOperator{\cone}{cone}

\DeclareMathOperator*{\argmin}{arg\,min}
\newcommand{\PP}{\mathbb{P}}

\newcommand{\RR}{\mathbb{R}}

\newcommand{\EE}{\mathbb{E}}
\newcommand{\e}{\mathbf{e}}
\newcommand{\uu}{\mathbf{u}}
\newcommand{\vv}{\mathbf{v}}
\newcommand{\zz}{\mathbf{z}}
\newcommand{\yy}{\mathbf{y}}
\newcommand{\qq}{\mathbf{q}}
\newcommand{\ww}{\mathbf{w}}
\newcommand{\xx}{\mathbf{x}}
\newcommand{\cc}{\mathbf{c}}

\begin{document}

\author{Mateo D\'iaz}
\address{
Center for Applied Mathematics\\
Cornell University\\
Ithaca, NY 14853, USA.
}
\email{md825@cornell.edu}
 
\author{Mauricio Junca}
\address{
Departamento de Matem\'aticas\\
Universidad de los Andes\\
Bogot\'a, Colombia.
}
\email{mj.junca20@uniandes.edu.co}

\author{Felipe Rinc\'on}
\address{
Department of Mathematics\\
University of Oslo\\
Oslo, Norway.
}
\email{feliperi@math.uio.no}

\author{
Mauricio Velasco\\ 
}
\address{
Departamento de Matem\'aticas\\
Universidad de los Andes\\
Bogot\'a, Colombia.
}
\email{mvelasco@uniandes.edu.co}


\title{Compressed sensing of data with a known distribution}

\begin{abstract} 
Compressed sensing is a technique for recovering an unknown sparse signal from a small number of 
linear measurements. When the measurement matrix is random,
the number of measurements required for perfect recovery exhibits a phase transition: there is a threshold on the number of measurements after which the probability of exact recovery quickly goes from very small to very large. 
In this work we are able to reduce this threshold by incorporating statistical information about the data we wish to recover. Our algorithm works by minimizing a suitably weighted $\ell_1$-norm, where the weights are chosen so that the expected statistical dimension of the corresponding descent cone 
is minimized. We also provide new discrete-geometry-based Monte Carlo algorithms for computing intrinsic volumes of such descent cones, allowing us to bound the failure probability of our methods.  
\end{abstract}

\maketitle

\section{Introduction}

The sensing problem consists on trying to recover a signal $\mathbf{x}_0 \in \RR^d$ from $m$ linear measurements encoded in a vector $\mathbf{y}_0 := \mathbf{Ax}_0$, where $\mathbf{A}$ is a given $m \times d$ matrix with $m < d$. In the seminal works by Cand\`es, Romberg, and Tao~ \cite{Candes2005,Candes2006} and Donoho~\cite{Donoho06}, the following convex optimization algorithm is proposed as a possible solution:
\begin{equation}
\tag{P}
\label{problem:1}
\arraycolsep=2pt\def\arraystretch{1.5}
\begin{array}{lr}
\Delta(\mathbf{y}_0) := \argmin\limits_{\mathbf{x} \in \RR^d} \|\mathbf{x}\|_1 & \text{ s.t. }\mathbf{Ax} = \mathbf{y}_0.
\end{array}
\end{equation}

We say that the problem \eqref{problem:1} is \textit{successful} or that it performs a \textit{perfect recovery} for $\mathbf{A}$ and $\mathbf{x}_0$ if it has a unique solution and this solution is $\mathbf{x}_0$.  We cannot expect this method to work for arbitrary signals and measurements; by taking $m$ strictly less than $d$ we are collapsing dimensions and consequently losing information. However, if $\mathbf{A} \in \RR^{m \times d}$ is a random matrix with independent Gaussian entries, it is shown in \cite{Candes2006,Donoho06} that this method is successful with very high probability for all sufficiently sparse vectors, i.e., vectors with a low number of non-zero entries. 

These success guarantees were obtained by proving that matrices with Gaussian entries satisfy the so-called Restricted Isometry Property with high probability (for suitable choices of $m$ and $d$), and that this condition is sufficient to guarantee that \eqref{problem:1} is successful for all sufficiently sparse vectors $\mathbf{x}_0$. 

Although the Restricted Isometry Property is a sufficient condition for \eqref{problem:1} to be successful, it does not explain the phase transition phenomenon exhibited by the probability of perfect recovery: If the number of measurements exceeds certain level related to the sparsity of the signal, exact recovery is obtained with very high probability, and if the number of measurements is below this level, exact recovery occurs with very small probability. Much of the later work has thus focused on understanding the geometry behind this phase transition phenomenon. We now discuss some of these results in detail, as they are relevant for the approach taken in this paper.

\begin{definition}{(\textbf{Descent Cone})} For a point $\mathbf{x}_0 \in \RR^d$ and $f: \RR^d \rightarrow \RR$ a convex function, the descent cone $D(f, \mathbf{x}_0)$ of $f$ at $\mathbf{x}_0$ is given by 
\[D(f,\mathbf{x}_0) := \cone\{\mathbf{x}-\mathbf{x}_0 : f(\mathbf{x}) \leq f(\mathbf{x}_0) \}.\]
\end{definition} 
In other words, the descent cone is the cone generated by all directions from $\mathbf{x}_0$ in which $f$ decreases. The following simple but very useful theorem relates the geometry of descent cones to the recovery problem.  
\begin{theorem}{(\cite{cha})}\label{theorem:characterization} The compressed sensing method \eqref{problem:1}  is successful for $\mathbf{A}$ and $\mathbf{x}_0$ if and only if $D(\|\cdot\|_1 , \mathbf{x}_0) \cap \ker(\mathbf{A}) = \{\mathbf{0}\}$.
\end{theorem}

If $\mathbf{A} \in \RR^{m \times d}$ is a random matrix with i.i.d. entries N$(0,1)$ then $\ker (\mathbf{A})$ is uniformly distributed over the Grassmannian $\text{Gr}(d-m,\RR^d)$ of $(d-m)$-dimensional subspaces of $\RR^d$ (see for instance \cite{Goldstein}). It follows that if $K$ is a fixed subspace of dimension $d-m$ and $\mathbf{Q} \in \text{O}(d)$ is a random rotation matrix, chosen with the Haar measure on $\text{O}(d)$, then  
\begin{align*}
    \PP_\mathbf{A} \{\text{\eqref{problem:1} is successful for }\mathbf{x}_0\text{ and }\mathbf{A}\}
    & = \PP_\mathbf{Q}\left\{ D(\|\cdot\|_1 , \mathbf{x}_0) \cap \mathbf{Q}K = \{\mathbf{0}\} \right\}.
\end{align*}

The problem of computing the probability that a random subspace intersects a given cone is a central problem in integral geometry, and there are explicit formulas for this probability in terms of the so-called intrinsic volumes. For the purposes of this paper we will only describe the basic concepts related to our problem, and we refer the reader to \cite{schneider2008,Ame15} for a more extended treatment. We follow the point of view presented in \cite{Ame}, where the concept of statistical dimension of a cone is introduced as the key invariant to understand phase transitions. An alternative approach to phase transitions is presented in \cite{cha} in terms of Gaussian widths of descent cones. Such phase transitions can also be interpreted as a qualitative change in the facial structure of polyhedra (crosspolytopes or simplices) under random projections (see \cite{dontan} and \cite{Ver} for details).

Let $C \subseteq \RR^d$ be any closed convex set, and let $\mathbf{x}\in \RR^d$. The projection of $\mathbf{x}$ onto $C$ is $\pi_C(\mathbf{x}) := \argmin\left\{\|\mathbf{x} - \mathbf{y}\|_2 : \mathbf{y} \in C\right\}$.

\begin{definition}{(\textbf{Intrinsic Volumes})} Let $C$ be a polyhedral cone in $\RR^d$. For each $0 \leq k \leq d$, the $k$th intrinsic volume $\nu_k(C)$ is given by 
\begin{equation*}
\nu_k(C) := \PP_ \mathbf{g}\{ \pi_C (\mathbf{g}) \text{ lies in the interior of a } k \text{-dimensional face of } C\},
\end{equation*}
where $\mathbf{g}$ is a standard normal random vector in $\RR^d$.
\end{definition}

There are many interesting results concerning intrinsic volumes, and also many open questions; see \cite{schneider2008,Ame,Ame15}. The intrinsic volumes of a cone $C$ give rise to a discrete probability measure over the set $\{0,1,\dots, d\}$. The expected value of a random variable with this distribution is called the statistical dimension of $C$.
\begin{definition}{(\textbf{Statistical Dimension})} Let $C\subseteq \RR^d$ be a polyhedral cone. The statistical dimension $\delta(C)$ is
\begin{equation*}
\delta(C) := \sum_{k=0}^d k \nu_k(C).
\end{equation*} 
\end{definition}
For many theoretical results it is useful to have the following equivalent definition. 
\begin{lemma}{(\cite{Ame})}
If $C\subseteq \RR^d$ is a polyhedral cone then its statistical dimension is equal to
\begin{equation*}
\delta(C) := \EE_ \mathbf{g}[\|\pi_C(\mathbf{g})\|_2^2],
\end{equation*} 
where $\mathbf{g}$ is a standard normal random vector in $\RR^d$.
\end{lemma}

The statistical dimension of the cone $D(\|\cdot\|_1 , \mathbf{x}_0)$ seems to be very close to the inflection point of the phase transition. A similar observation was made in \cite{cha} using the concept of Gaussian width. 
More precisely, given a convex function $f: \RR^d \rightarrow \RR$ and $\mathbf{x}_0 \in \RR^d$, consider the convex optimization problem
\begin{equation}
\tag{P$_f$}
\label{problem:f}
\arraycolsep=2pt\def\arraystretch{1.5}
\begin{array}{lr}
\min\limits_{\mathbf{x} \in \RR^d} f(\mathbf{x}) & \text{ s.t. }\mathbf{Ax} = \mathbf{Ax}_0.
\end{array}
\end{equation}

\begin{theorem}{(\cite{Ame})}\label{teo:mainAme}
Fix a tolerance $\eta \in (0,1)$. Let $\mathbf{x}_0  \in \RR^d$, $f: \RR^d \rightarrow \RR$ a convex function, and $\mathbf{A} \in \RR^{m \times d}$ a random matrix with i.i.d. entries N$(0,1)$. Then, 
$$
\begin{array}{rcl}
m \leq \delta(D(f, \mathbf{x}_0)) - a_\eta \sqrt{d} & \Longrightarrow  & \text{\eqref{problem:f} succeeds with probability} \leq \eta \\
m \geq \delta(D(f, \mathbf{x}_0)) + a_\eta \sqrt{d} & \Longrightarrow & \text{\eqref{problem:f} succeeds with probability} \geq 1 - \eta,
\end{array}
$$
where $a_\eta := \sqrt{8\log(4/\eta)}.$
\end{theorem} 

The proof of this theorem is based on the \textit{kinematic formula} from integral geometry, which relates the probability of success with the intrinsic volumes of the corresponding descent cone.

\begin{definition} \label{def:tails}
Let $C \subseteq \RR^d$ a closed convex cone. For each $k \in \{0,1, \dots, d\}, $ the $k$th tail functional is defined as
\[t_k(C) := \sum\limits_{j=k}^d \nu_j(C).\]
Similarly, the $k$th half-tail functional is defined as 
\[h_k(C) := \sum\limits_{\substack{j=k \\ j -k  \text{ even}}}^d \nu_j(C).\] 
\end{definition}

\begin{theorem}{(\textbf{Kinematic formula} \cite{Ame15})} 
Let $C \subseteq \RR^d$ a closed convex cone and $L \subseteq \RR^d$ a linear subspace of dimension $d-m$. Then 
$$\PP_\mathbf{Q}\left\{ C \cap \mathbf{Q}L = \{\mathbf{0}\} \right\} = 1 - 2h_{m+1}(C).$$
\end{theorem}

\begin{remark} It is shown in \cite{Ame} that for each closed convex cone $C \subseteq \RR^d$ that is not a linear subspace, 
\[2h_k(C) \geq t_k(C) \geq 2 h_{k+1}(C) \hspace{0.8cm} \text{for }k = 0,1,2, \cdots, d-1.\] Thus, combining the previous two results, we conclude that $1-t_m(D(\|\cdot\|_1, \mathbf{x}_0))$ is very close to the exact probability of perfect recovery. We can therefore think of the probability of failure as a ``tail estimate" of the distribution of intrinsic volumes.
\end{remark}
Unfortunately, even for polyhedral cones such as $D(\|\cdot\|_1 , \mathbf{x}_0)$, there is no simple closed formula for intrinsic volumes. One of the main contributions of this work is to present an efficient algorithm to estimate these intrinsic volumes via a combination of discrete geometry and Monte Carlo simulations (Section \ref{section: Formulas}).

\subsection*{Our aim: compressed sensing and statistics.} In the most common setup for Compressed Sensing, the only assumption about the signal $\mathbf{x}_0$ is that it is sparse. However, in real applications this signal is often a random vector whose distribution can be approximated by using its previous realizations, i.e., historical information on the signal. It is therefore natural to ask: Can we modify the compressive sensing paradigm to use to our advantage this extra distributional information? This new setting opens the possibility of further reducing the number of measurements $m$ needed for perfect recovery.

The main contribution of this article is to answer this question affirmatively by weighting the norm used in the recovery procedure. More specifically, we consider the optimization problem 
\begin{equation}
\tag{P$_{\mathbf w}$}
\arraycolsep=2pt\def\arraystretch{1.5}
\begin{array}{lr}
\Delta_\mathbf{w}(\mathbf{y}_0) = \argmin\limits_{\mathbf{x}\in \RR^d}  \|\mathbf{x}\|_1^\mathbf{w} &\text{ s.t. }\mathbf{Ax} = \mathbf{y}_0.
\end{array}
\end{equation}
where $\|\mathbf{x}\|^\mathbf{w}_1 = \sum\limits_{i=1}^d w_i|x_i|$ and $\ww\in \RR^d$ is a vector of weights to be specified.

We study the question of how to choose weights in order to reduce the number $m$ of measurements necessary for perfect recovery. Our main contributions are organized as follows:
\begin{enumerate}
\item In Section~\ref{section:weighted} we introduce the idea of random descent cones and the related concepts of expected statistical dimension and expected intrinsic volumes. We show in Theorem~\ref{theorem:concentratedDimensions} that the expected statistical dimension $\overline{\delta}(\ww)$ plays a
key role in phase transitions of compressive sensing of random signals. 

\item Motivated by $(1)$, in Theorem~\ref{teo:ProtoOptimal} we give analytic formulas, depending on the distribution of the signal $\mathbf{x}_0$, for weights which minimize an upper bound on the expected statistical dimension. 

\item Aiming to analyze the quality of the bound in $(2)$, in Section~\ref{section: Formulas} we study the geometry of the descent cones for the weighted norms $\|\cdot\|^\mathbf{w}_1$. In Theorem~\ref{thm: main} we give closed formulas for the projection maps onto such cones. Using these explicit formulas, we introduce new Monte Carlo algorithms for efficiently computing the expected statistical dimension and the failure probabilities for weighted compressed sensing. 

These Monte Carlo algorithms allow us to see that the upper bound obtained in $(2)$ and the actual value of the expected statistical dimension are often very close (see Section ~\ref{Bernoulli}). We also discover that there are distributions for which minimizing the expected statistical dimension underperforms the unweighted approach in some regimes (see Section~\ref{NonSharp} for an example).

\item In Section~\ref{section:DescentAlgo} we give a second discrete-geometry based Monte Carlo algorithm to search for weights which are local minima of the expected statistical dimension, via stochastic gradient descent. 

\item Finally, in Section \ref{Applications} we test the performance of our algorithms in a few numerical examples. The examples suggest that the weighted approach is often superior to the unweighted approach. Our experiments suggest that this occurs in cases of practical interest such as brain MRI data. Section~\ref{section:conclusions} contains the conclusions and discusses a few open problems.
\end{enumerate}


\subsection*{Related work}

Compressed sensing with prior information has been studied in the past under different models. In each model the known information is different. For example, the paper \cite{mota2014compressed} analyzes the case where a signal similar to the one to be recovered is known beforehand. Also, \cite{vaswani2010modified, friedlander2012recovering} assume information about the support of the signal $\mathbf{x}_0$. Specifically, the first paper assumes to know part of the support entries, and the second one assumes to have prior knowledge about the support location. All these analyses mainly rely on the Restricted Isometry Property, an approach that we do not pursue in this work.  

The idea of using weighted $\ell_1$-minimization in this subject was first introduced in \cite{candes2008enhancing}. This paper proposes an iterative algorithm with dynamic weights to reduce the number of necessary measurements to recover the signal. Weighted $\ell_1$-minimization has also been investigated under probabilistic hypotheses. The papers \cite{xu2010compressive, khajehnejad2009weighted} consider the case in which the allowed non-zero entries of the vector fall into two sets, and each set has a different probability of being non-zero; \cite{khajehnejad2011analyzing} generalizes this model to $n$ sets (our numerical example \ref{Bernoulli} falls into this setting). The article  \cite{misra2015weighted} considers a Bayesian setting, where the entries are independent and the probability of being non-zero is given by a continuous function. This paper independently obtained a result analogous to Theorem \ref{teo:ProtoOptimal} below, but our methods give the first tools to systematically explore their quality. The study done in this work is based on Grassmmann angles; our approach using intrinsic volumes is implicitly related to these concepts. 

Weighted $\ell_1$-minimization has also recently been used in other contexts. For instance, \cite{rauhut2015interpolation} develops a theory about the use of weighted norms to better interpolate smooth functions that are also sparse. Similarly, \cite{schiebinger2015superresolution} uses a weighting function $w(\cdot)$ to tackle the problem of superresolution imagining. In particular, the paper studies how to recover a point measure that encodes a signal by using a convex algorithm and a relatively small set of measurements for this signal.

\section{Compressed sensing with a priori distributions}\label{section:weighted}

Let $\mathbf{X}_0\in \RR^d$ be the data we wish to recover and assume that $\mathbf{X}_0$ follows some known distribution $\mathcal{F}$. In order to reduce the number of necessary measurements $m$, we aim to increase the probability of \eqref{problem:1} being successful. Theorem \ref{theorem:characterization} gives us a good insight on what we can do. Imagine that we have a particular point $\mathbf{x'}$ with very high probability. Then, if we can modify the $\ell_1$-norm in order to sharpen the descent cone at $\mathbf{x'}$ and reduce the probability of intersection with the kernel of $\mathbf{A}$, we should be able to increase the probability of perfect recovery; see Figure \ref{fig:conos}. How to obtain a good modification of the $\ell_1$-norm based on the distribution $\mathcal{F}$ is explained in this section.

\begin{figure}[htb]
\begin{centering}
        \begin{subfigure}[b]{0.3\textwidth}
                \includegraphics[width=\textwidth]{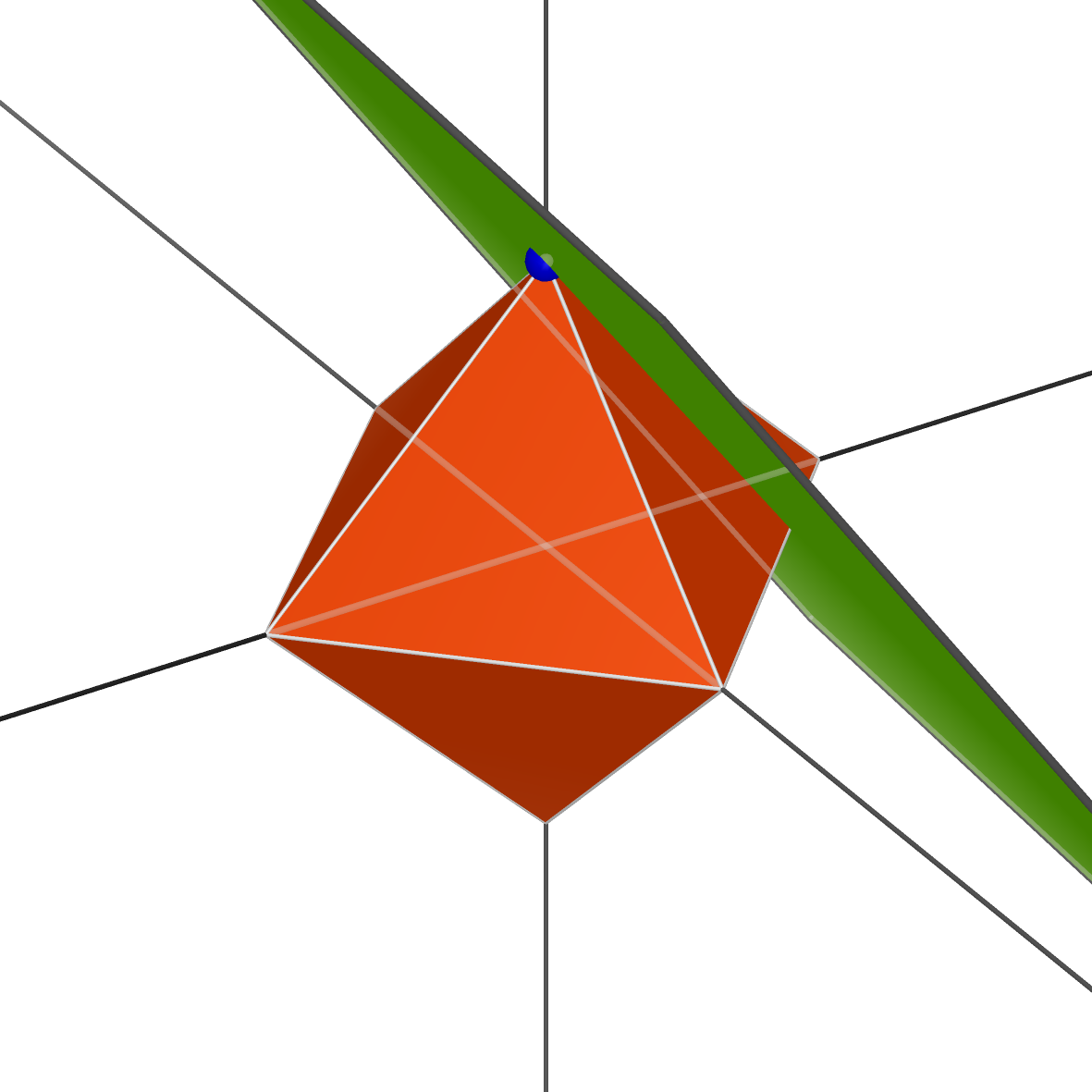}
                \label{fig:polytope1}
        \end{subfigure}
        \qquad\qquad
        \begin{subfigure}[b]{0.3\textwidth}
                \includegraphics[width=\textwidth]{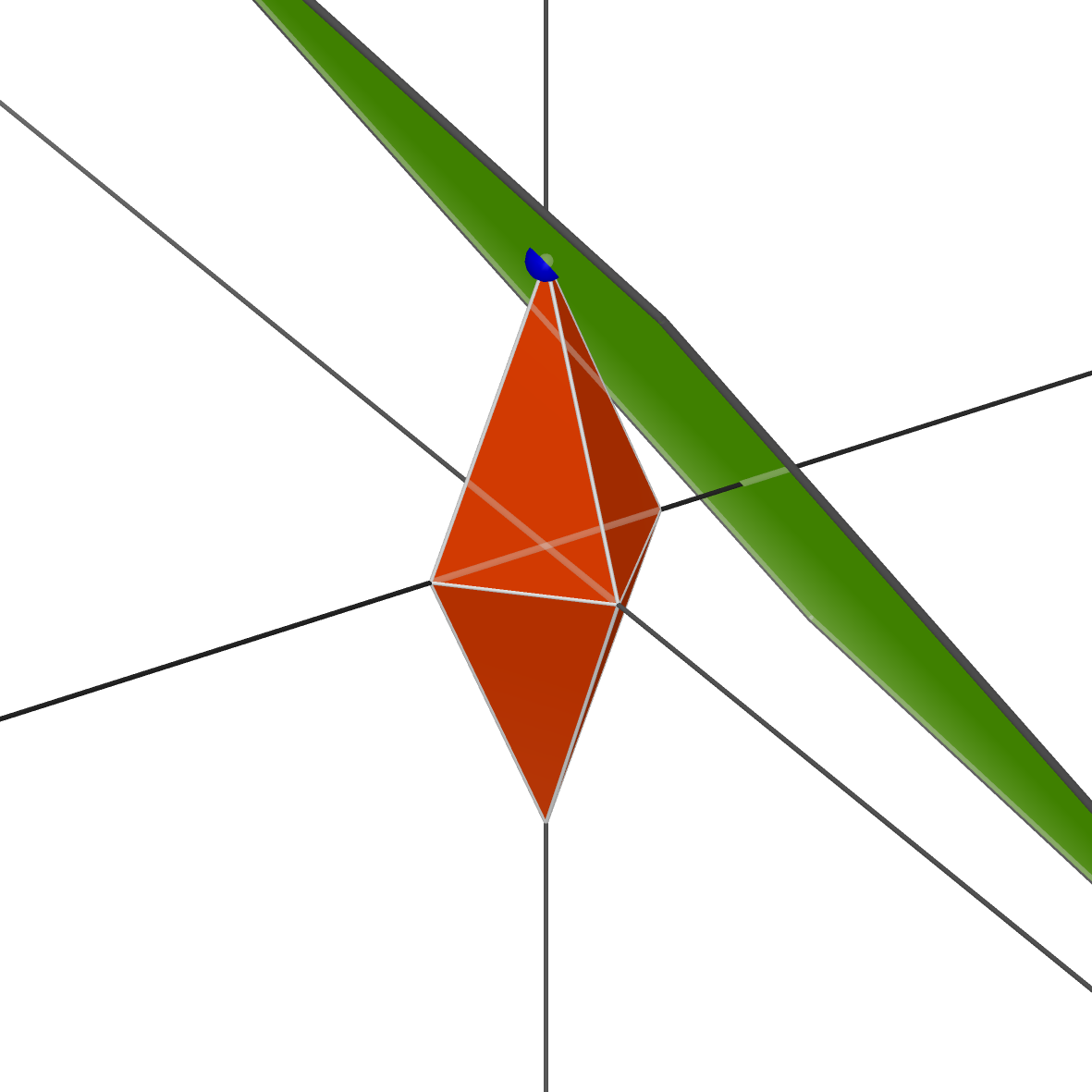}
                \label{fig:polytope2}
        \end{subfigure}
        \caption{The blue points are $\mathbf{x'}$, the polytopes are the $\ell_1$ and $\ell_1^{\mathbf w}$ balls and the green planes are the $\ker (\mathbf{A})+\mathbf{x'}$.}\label{fig:conos}
\end{centering}
\end{figure}

\begin{definition}{(\textbf{$\mathbf{w}$-weighted $\ell_1$-norm})}
For a fixed vector $\mathbf{w} \in \RR_{> 0}^d$, the $\ell_1^{\mathbf w}$-norm is given by
$$\|\mathbf{x}\|^\mathbf{w}_1 = \sum\limits_{i=1}^d w_i|x_i|.$$
\end{definition}

Note that by modifying the weights we are able to deform the descent cones. Consider the optimization problem 
\begin{equation}
\tag{P$_{\mathbf w}$}
\label{problem:w}
\arraycolsep=2pt\def\arraystretch{1.5}
\begin{array}{lr}
\Delta_\mathbf{w}(\mathbf{y}_0) = \argmin\limits_{\mathbf{x}\in \RR^d}  \|\mathbf{x}\|_1^\mathbf{w} &\text{ s.t. }\mathbf{Ax} = \mathbf{y}_0.
\end{array}
\end{equation}

We say that the problem \eqref{problem:w} is \textit{successful} or performed a \textit{perfect recovery} for $\mathbf{A}$ and $\mathbf{x}_0$ if it has a unique solution and this solution is $\mathbf{x}_0$, where $\mathbf{y}_0 := \mathbf{Ax}_0$. 

\begin{definition}
Let $\mathbf{X}_0 \sim \mathcal{F}$ be a random vector in $\RR^d$ and $\mathbf{A} \in \RR^{m \times d}$ a random matrix with i.i.d. entries N$(0,1)$ independent of $\mathbf{X}_0$. For a given vector $\mathbf{w} \in \RR_{> 0}^d$ we define the success probability as
    $$s(\mathbf{w})=\PP_\mathbf{A,X_0} \{\text{\eqref{problem:w} is successful for }\mathbf{A}\text{ and }\mathbf{X}_0\}.$$
\end{definition}

Let  $[d] := \{1,2,\dotsc, d\}$. If $\mathbf{x} \in \RR^d$, its support is $\supp(\mathbf{x}):= \{i \in [d] : x_i \neq 0\}$. If $B$ is the $\ell_1^{\mathbf w}$ ball of radius $|\!|\mathbf x|\!|_1^{\mathbf w}$, the support of $\mathbf{x}$ determines which face of $B$ the vector $\mathbf x$ lies in. It follows that all vectors with a fixed support $I \subseteq [d]$ have the same descent cone $D(\|\cdot\|_1^\mathbf{w} , \mathbf{x})$. 
We will thus adopt the notation $D(I, \mathbf{w}) := D(\|\cdot\|_1^\mathbf{w}, \mathbf{x})$ where $\mathbf{x}$ has support $I$. In this way we focus on the distribution induced by $\mathcal{F}$ over the subsets $I$ of $[d]$. 

\begin{definition}{(\textbf{Expected intrinsic volumes})}
For a fixed vector $\mathbf{w} \in \RR_{> 0}^d$ and $\mathbf{X}_0 \sim \mathcal{F}$ a random vector, we consider the random descent cone $D\left(\supp (\mathbf{X}_0), \mathbf{w}\right)$ and define the $k$th expected intrinsic volume as
\[ \bar{\nu}_k(\mathbf{w}) = \EE_{\mathbf{X}_0}\left[ \nu_k(D\left(\supp (\mathbf{X}_0), \mathbf{w}\right))\right],\]
for $k = 0, \dotsc ,d$.
\end{definition}

We define $\bar{t}_k$ and $\bar{h}_k$ as the tail and the half-tail of the expected intrinsic volumes, just as in Definition \ref{def:tails}. It is easy to prove by conditioning that the failure probability is given by 
\[1 - s(\mathbf{w}) = 2\bar{h}_{m+1}(\mathbf{w}).\]

Since even in the deterministic case this probability is very hard to compute, in the next sections we will use our Monte Carlo algorithm to estimate it. We now concentrate on the behavior of the phase transition in this setting. We begin by stating a phase-transition Theorem analogous to~\ref{teo:mainAme} in this setting. 
\begin{theorem}\label{theorem:concentratedDimensions}
Fix a tolerance $\eta \in (0,1)$. Let $\mathbf{w} \in \RR_{> 0}^d$ be a fixed vector of weights, $\mathbf{X}_0 \sim \mathcal{F}$ a random vector, and $\mathbf{A} \in \RR^{m \times d}$ a random matrix with i.i.d. entries N$(0,1)$ independent from $\mathbf{X}_0$. Then
$$\begin{array}{lcl}
\PP_{\mathbf{X}_0} \bigl\{ m \leq \delta(D(\supp (\mathbf{X}_0),\mathbf{w})) - a_{\eta/2} \sqrt{d}\bigr\} \geq 1 - \eta/2 & \Longrightarrow & s(\mathbf{w}) \leq \eta,\\
\PP_{\mathbf{X}_0} \bigl\{ m \geq \delta(D(\supp (\mathbf{X}_0),\mathbf{w})) + a_{\eta/2} \sqrt{d}\bigr\} \geq 1 - \eta/2 & \Longrightarrow & s(\mathbf{w}) \geq 1 - \eta,  
\end{array}$$
where $a_{\eta/2} := \sqrt{8\log(8/\eta)}.$
\end{theorem}
\begin{proof}
We begin with the first implication. For a fixed support $I \subseteq [d]$ we define the conditional probability 
$$s(I,\mathbf{w}) := \PP_\mathbf{A} \{\text{\eqref{problem:w} is successful for }\mathbf{A}\text{ and }\mathbf{x}\text{, where $\mathbf x$ has support }I\}.$$
Also, let $q_I := \PP_{\mathbf{X}_0}\{\supp(\mathbf{X}_0) = I\}$.
Fix an $m$, and let $\Gamma$ be the collection of supports containing all the subsets $I$ that satisfy 
\begin{equation}
\label{cond:TheTheorem} 
m \leq \delta(D(I,\mathbf{w})) - a_{\eta/2} \sqrt{d}.
\end{equation} 
By conditioning on the support we obtain
\[s(\mathbf{w}) = \displaystyle\sum\limits_{I\in\Gamma} q_I \, s(\mathbf{w},I) + \sum\limits_{I\in\Gamma^c} q_I \, s(\mathbf{w},I).\]
Now, by definition, any element of $\Gamma$ satisfies \eqref{cond:TheTheorem}, thus we may apply Theorem \ref{teo:mainAme} to bound $s(\mathbf{w},I)$ in the first sum with $\eta/2$ and in the second one with $1$. Then,
\begin{align*}
s(\mathbf{w}) & \leq \displaystyle\sum_{I\in\Gamma} q_I \, \dfrac{\eta}{2} + \sum_{I \in \Gamma^c} q_I\\
& \leq  \displaystyle\dfrac{\eta}{2} + \dfrac{\eta}{2} \\
& =  \eta.
\end{align*}
The last inequality follows from the hypothesis on $\Gamma$; namely, $\PP_{\mathbf{X}_0}\{\supp(\mathbf{X}_0)\in\Gamma^c\}\leq \eta/2.$ 
An analogous argument proves the second implication. 
\end{proof}

The above theorem implies that the more the distribution of the random variable $\delta(D(\supp (\mathbf{X}_0),\mathbf{w}))$ is concentrated around its mean, the sharper the phase transition will be. Hence, we introduce the following definition.

\begin{definition}{(\textbf{Expected statistical dimension})}
For a fixed vector $\mathbf{w} \in \RR_{> 0}^d$ and a $\mathbf{X}_0 \sim \mathcal{F}$ a random vector, the expected statistical dimension is given by 
\begin{equation}\label{def:expectedStatisticalDim}
\overline{\delta}(\mathbf{w}):=\EE_{\mathbf{X}_0}\left[\delta(D(\supp (\mathbf{X}_0),\mathbf{w}))\right].
\end{equation}
\end{definition}

Now, we propose to choose the weights $\mathbf{w}$ so as to minimize the expected statistical dimension. Theorem \ref{teo:ProtoOptimal} below provides an analytically tractable upper bound for $\overline{\delta}(\mathbf{w})$ and a description of the weights resulting of minimizing this bound. We begin with a lemma that allows us to bound the statistical dimension of the cones $D(I,\mathbf{w})$.

\begin{lemma}\label{lemma:bound} For any $I\subseteq [d]$ and $\mathbf{w} \in \RR_{> 0}^d$, the statistical dimension of the descent cone $D(I,\mathbf{w})$ satisfies
\[ \delta(D(I,\mathbf{w}))\leq \inf_{\tau\geq 0} \left( |I|+\tau^2\left(\sum_{i\in I} w_i^2\right) +\sum_{i\not\in I} \frac{2}{\sqrt{2\pi}} \int_{\tau w_i}^\infty (u-\tau w_i)^2e^{-\frac{u^2}{2}}du\right),\]
where the infimum is achieved at a unique $\tau>0$.
\end{lemma}
\begin{proof}
This is a special case of Proposition $4.1$ in \cite{Ame}. Recall that the subdifferential of a convex function $f: \RR^d \rightarrow \RR$ at a point $\mathbf{x}$ is defined as
\[\partial f(\mathbf{x}) := \{\mathbf{z} \in \RR^d : f(\mathbf{y}) \geq f(\mathbf{x}) + \langle \mathbf{z}, \mathbf{y} - \mathbf{x}\rangle \text{ for all }\mathbf{y} \in \RR^d\}.\]
If $\mathbf{x}\in \RR^d$ is such that $x_i >0$ if $i\in I$ and $x_i = 0$ otherwise, then for $\tau\geq 0$ the $\tau$ scaling of the subdifferential of $|\!|\cdot |\!|_1^{\mathbf w}$ at $\mathbf{x}$ is given by
\[\tau \partial |\!|\mathbf{x}|\!|_1^\mathbf{w}:=\left\{(z_1,\dots, z_d): 
\begin{cases}
z_i=\tau w_i &\text{if $i\in I$,}\\
|z_i|\leq \tau w_i &\text{if $i\not\in I$.}
\end{cases}
\right\}.
\]
Proposition $4.1$ in \cite{Ame} states that $\delta(D(I,\mathbf{w}))\leq \inf_{\tau\geq 0}\EE_{\mathbf{g}}\left[{\rm dist}^2(\mathbf{g},\tau\partial |\!|\mathbf{x}|\!|_1^{\mathbf w})\right]$, where $\mathbf{g}\in\RR^d$ is a standard normal vector, and that the infimum is achieved at a unique $\tau>0$. Since,
\begin{equation}\label{dist}
{\rm dist}^2(\mathbf{g},\tau\partial |\!|\mathbf{x}|\!|_1^{\mathbf w})=\sum_{i\in I} (g_i-\tau w_i)^2 +\sum_{i\not\in I} \left( (|g_i|-\tau w_i)^+\right)^2,
\end{equation}
where $(z)^+ = z$ if $z > 0$ and $(z)^+ = 0$ otherwise, taking expected value the statement follows.
\end{proof}

Recall that $q_I = \PP_{\mathbf{X}_0}\{\supp(\mathbf{X}_0) = I\}$ for $I\subseteq [d]$. If $i\in [d]$, let $\beta_i:=\sum_{I\ni i} q_I$. We assume that $0<\beta_i<1$ for all $i$.

\begin{theorem}\label{teo:ProtoOptimal} For any $\mathbf{w}\in \RR_{> 0}^d$ and any $\tau>0$ the following inequality holds
\begin{multline} 
\label{eq:firstIneq}
\overline{\delta}(\mathbf{w})\leq \EE_{\mathbf{X}_0}\left[|\supp(\mathbf{X}_0)|\right]+ \\ 
\sum_{j=1}^d \left(\beta_j(\tau w_j)^2+(1-\beta_j)\left[ \sqrt{\frac{2}{\pi}}\int_{\tau w_j}^\infty (u-\tau w_j)^2e^{-\frac{u^2}{2}}du \right]\right).
\end{multline}
The right hand side is minimized if $\lambda_i:=\tau w_i$ satisfy the equation
\begin{equation}\label{eq:formulaWeights}
  \lambda_i\, \frac{\beta_i}{ (1-\beta_i)}= \sqrt{\frac{2}{\pi}} \, \int_{\lambda_i}^{\infty} (u-\lambda_i)e^{-\frac{u^2}{2}}du.
\end{equation}  
\end{theorem}

\begin{proof} Conditioning on $\supp(\mathbf{X}_0)$ we have that
\[ \EE_{\mathbf{X}_0}[ \delta (D(\supp(\mathbf{X}_0),\mathbf{w}))]=\sum_{I\subseteq [d]} \delta(D(I,\mathbf{w}))\, q_I.\]
By Lemma \ref{lemma:bound} we know that the right hand side is bounded for any choice of $\tau_I>0$ by
\[ \sum_{I\subseteq [d]} q_I \left( |I|+\tau_I^2\left(\sum_{i\in I} w_i^2\right) +\sum_{i\not\in I} \frac{2}{\sqrt{2\pi}} \int_{\tau_i w_i}^\infty (u-\tau_I w_i)^2e^{-\frac{u^2}{2}}du\right).\]
In particular, the inequality holds when all $\tau_I$ coincide with a given value $\tau$. Changing the order of summations we conclude that $\EE_{\mathbf{X}_0}[ \delta (D(\supp(\mathbf{X}_0),\mathbf{w}))]$ is bounded above by
\[ \sum_I|I|\,q_I + \sum_{j=1}^d \left[\tau^2w_j^2\left( \sum_{I\ni j}q_I\right) +\left(\sum_{I\not\ni j} q_I\right)\left( \sqrt{\frac{2}{\pi}} \int_{\tau w_j}^\infty (u-\tau w_j)^2e^{-\frac{u^2}{2}}du\right) \right],\]
which proves the first claimed inequality. Now, writing the right-hand side of \eqref{eq:firstIneq} as a function of $\boldsymbol\lambda$ and using \eqref{dist} we obtain
\begin{align*}
h(\boldsymbol\lambda):&=\EE_{\mathbf{X}_0}\left[|\supp(\mathbf{X}_0)|\right]+ \sum_{j=1}^d \left(\beta_j\lambda_j^2+(1-\beta_j)\left[ \sqrt{\frac{2}{\pi}}\int_{\lambda_j}^\infty (u-\lambda_j)^2e^{-\frac{u^2}{2}}du \right]\right)\\
&=\EE_\mathbf{g,X_0}\left [\sum_{i\in \supp(\mathbf{X}_0)} (g_i-\lambda_i)^2 + \sum_{i\not\in \supp(\mathbf{X}_0)} ((|g_i|-\lambda_i)^+)^2\right],
\end{align*}
where $\mathbf{g}$ is a normally distributed random vector in $\RR^d$. Since the function inside the expectation is convex in $\boldsymbol\lambda$, $h(\boldsymbol\lambda)$ is also convex. It follows that $h(\boldsymbol\lambda)$ is minimized at any point with $\nabla h(\boldsymbol\lambda)=0$. The equation $\frac{\partial h}{\partial \lambda_i}=0$ is equivalent to
\begin{equation*}
  \lambda_i\, \frac{\beta_i}{ (1-\beta_i)}= \sqrt{\frac{2}{\pi}} \, \int_{\lambda_i}^{\infty} (u-\lambda_i)e^{-\frac{u^2}{2}}du,
\end{equation*}  
and thus \eqref{eq:formulaWeights} holds.
\end{proof}

The paper \cite{misra2015weighted} independently found a very similar result. In this paper, Misra and Parrilo consider a Bayesian setting where the entries are independent and the probability of being non-zero is given by a continuous function, i.e., $\beta_i = p(i/d)$. If one takes a discrete measure to integrate, then the two methods find the same optimal weights. The useful feature about our formulation is the fact that equation \eqref{eq:formulaWeights} allows us to use a simple binary search algorithm to efficiently find these weights. 

\section{Estimating intrinsic volumes for weighted crosspolytopes}
\label{section: Formulas}

In order to estimate the quality of the bound in Theorem \ref{teo:ProtoOptimal} we need to calculate the actual expected statistical dimension. A method to do so is to approximate the corresponding intrinsic volumes, and for that we need to understand the projections onto descent cones $D(I, \mathbf{w})$ arising from weighted norms. This is the main objective of this section. We begin by deriving formulas for $\pi_C$ when $C$ is the descent cone generated by a weighted crosspolytope. These formulas will allow us to determine the dimension of the unique face of $C$ whose relative interior contains $\pi_C(\zz)$ for any $\zz$, and in particular, to propose an efficient Monte Carlo method to estimate the intrinsic volumes of $C$.

Let $\e_0,\dots,\e_d$ be the canonical basis of $\RR^{d+1}$ and fix positive weights $w_1,\dots, w_d$. Let $S:\RR^{d+1}\rightarrow \RR^{d+1}$ be the map $S(x_0,x_1,\dots,x_d)=(x_0,w_1x_1,\dots, w_dx_d)$. If $\sigma: \RR^{d+1}\rightarrow \RR^{d+1}$ acts by permuting the last $d$ components, define the weighted permutation $\sigma_w:=S^{-1}\circ \sigma \circ S$.

For a real number $a>0$ let $C={\rm cone}\{\e_0/a\pm \e_i/w_i:i=1,\dots d\}$. The cone $C$ is the cone over a $d$-dimensional crosspolytope. For $i=1,\dots, d$ define $\uu_i =\e_0/a + \e_i/w_i$ and $\vv_i = -a\e_0+\sum_{r=1}^i w_r\e_r -\sum_{r=i+1}^{d}w_r\e_r$. 

\begin{lemma} \label{lem: basic} The following statements hold:
\begin{enumerate}
\item $C$ is invariant under sign changes and weighted permutations $\sigma_w$ of the last $d$ components.
\item The dual cone $C^*:=\{\yy\in \RR^{d+1}:\forall \xx\in C(\yy^t\xx\leq 0)\}$ is given by
\[C^* = {\rm cone}\{-a\e_0\pm w_1\e_1 \pm w_2\e_2\pm \dots \pm w_d \e_d\}.\] 
In particular $C^*$ is combinatorially equivalent to a cone over the $d$-dimensional cube.
\item Up to weighted permutations and sign changes of the last $d$ components, every $k$-dimensional face of $C$ is of the form $F={\rm cone}\{\uu_1,\dots, \uu_k\}$, for $1 \leq k \leq d$. In particular, every proper face of $C$ is a simplicial cone.
\item The face $F^{\vee}:=\{\yy\in C^*: \forall \xx\in F(\yy^t\xx=0)\}$ dual to $F = {\rm cone}\{\uu_1,\dots, \uu_k\}$ is given by 
\[\textstyle F^{\vee}={\rm cone}\{-a\e_0+\sum_{i=1}^k w_i\e_i\pm w_{k+1}\e_{k+1}\dots\pm w_d\e_d\}.\]
\end{enumerate}
\end{lemma}
\begin{proof} $(1)$ Changing the sign of any of the last $d$ components or applying any weighted permutation $\sigma_w$ only permutes the generators of $C$ and therefore fixes the cone. $(2)$ If $(y_0,\dots, y_d)\in C^*$ then $y_0/a\pm y_i/w_i\leq 0$ so $|y_i/w_i|\leq -ay_0$ for $i=1,\dots, d$. This is a rescaling of the usual inequalities defining the cone over a cube. As a result the extreme rays of the cone defined by these inequalities are given by rescaling the vertices of the cube, and are therefore of the form $-a\e_0\pm w_1\e_1 \pm w_2\e_2\pm \dots \pm w_d \e_d$ as claimed. $(3)$ A proper face of $C$ cannot contain two opposite rays $\e_0/a + \e_i/w_i$ and $\e_0/a - \e_i/w_i$ of $C$, as then it would contain the interior point $2\e_0/a$. Up to permutations and sign changes from $(1)$, every subset of the generators of $C$ not containing opposite rays is of the form $\{\uu_1,\dots, \uu_k\}$ for some $k=0,\dots, d$. Every such set is obviously linearly independent. We will show that $F$ is a face of $C$ by verifying that the element $\tau= -a\e_0+w_1\e_1+\dots +w_k\e_k$ is an element of $C^*$ which vanishes precisely at the claimed generators. This is because the dot products $\tau^t\left(\e_0/a\pm \e_i/w_i\right)=-1\pm \tau^t \e_i/w_i$ take values in $\{-2,-1,0\}$, and equal $0$ if and only if the sign is positive and $i\leq k$. $(4)$ The face $F^{\vee}$ is generated by the extreme rays of $C^*$ which have vanishing dot product with $\uu_1,\dots, \uu_k$. The claim is therefore immediate from the list of extreme rays of $C^*$ computed in part $(2)$.
\end{proof}

Given $\zz\in \RR^{d+1}$ we would like to find an expression $\zz = \cc+\cc'$ with $\cc$ in the relative interior of a face $F\subseteq C$ and $\cc'$ an element of the face $F^{\vee}\subseteq C^*$, as this certifies that $\pi_C(\zz)=\cc$ and gives us a formula for the projection. 
Note that unlike $F$, the face $F^{\vee}$ is generally {\it not} simplicial. However, as shown in the previous lemma, the faces of $C^{*}$ are cones over hypercubes, and every hypercube admits a natural decomposition into simplices: the hypercube $|x_i|\leq 1$ in $\RR^d$ decomposes into $d!$ simplices which are the images of $P=\{x: 1\geq x_1\geq x_2\geq \dots\geq x_d\geq 0\}$ under all permutations of the components. This decomposition is the geometric motivation for the following key lemma.

\begin{lemma}\label{lem: key} If $\mathbf{z}\in \RR^{d+1}$ then the following statements hold:
\begin{enumerate}
\item For any integer $m=1,\dots, d$ the vectors $\mathbf{u}_1,\dots, \mathbf{u}_m, \mathbf{v}_m,\dots \mathbf{v}_d$ are a basis for $\RR^{d+1}$. 
\item Assume $z_1/w_1\geq \dots\geq z_d/w_d\geq 0$ and define $b_1\leq b_2\leq \dots \leq b_d$ by the formula
\[b_j:=\begin{cases}
\sum_{i=1}^j w_iz_i -\frac{\left(a^2+\sum_{i=1}^jw_i^2\right)}{w_{j+1}}z_{j+1} &\text{if $1\leq j\leq d-1$,}\\
\sum_{i=1}^d w_iz_i &\text{if $j=d$.}
\end{cases}
\]
If $l$ is the unique integer for which the inequalities $b_{l-1}<az_0\leq b_l$ hold (here by convention we set $b_{-1}=-\infty$ and $b_{d+1}=\infty$) then letting $t:=\frac{-az_0+\sum_{i=1}^l w_iz_i}{a^2+\sum_{i=1}^l w_i^2}$ and $\alpha_i:=w_iz_i-w_i^2t$ we have 
\begin{enumerate}
\item The projection of $\zz$ towards $C$ lands in the relative interior of the $l$-dimensional face $F_l = {\rm cone}\{\uu_1,\dots, \uu_l\}$, and is given by the formula
\[\pi_C(\zz)=\sum_{i=1}^l \alpha_i\uu_i\]

\item The following equality holds
\[\|\pi_C(\zz)\|^2=(z_0+at)^2+\sum_{i=1}^l (z_i-w_it)^2.\]
\end{enumerate}
 
\end{enumerate}
\end{lemma}
\begin{proof}$(1)$ For any integer $m$ with $1\leq m\leq d$ and any real numbers $\alpha_1,\dots, \alpha_m$, $\beta_m,\dots, \beta_d$ let $s(\alpha)=\sum_{i=1}^m \alpha_i$ and $t(\beta)=\sum_{i=m}^d\beta_i$. If $\zz=\sum_{i=1}^m \alpha_i \uu_i +\sum_{i=m}^d\beta_i \vv_i$,  then the following equalities hold
\[
z_i=\begin{cases}
s(\alpha)/a-at(\beta) &\text{if $i=0$,}\\
\alpha_i/w_i +w_it(\beta) &\text{if $1\leq i\leq m$,}\\
w_i(-\beta_{m}-\dots-\beta_{i-1}+\beta_{i}+\dots+\beta_d) &\text{if $m+1\leq i\leq d$.}
\end{cases}
\]
As a result $\alpha_i = w_iz_i-w_i^2t(\beta)$ for $i=1,\dots, m$. 
Replacing these expressions for $\alpha_i$ in the equation for $z_0$ we conclude that the equality 
\[ az_0 = \sum_{i=1}^m w_iz_i-t(\beta)\left(a^2+\sum_{i=1}^mw_i^2\right)\] 
holds, obtaining a formula for $t(\beta)$ in terms of the components of $\zz$:
\[t(\beta)=\frac{-az_0+\sum_{i=1}^m w_iz_i}{a^2+\sum_{i=1}^m w_i^2}.\]
Combining this formula with the expressions for $z_{m+1},\dots, z_d$ above, we find that the remaining coefficients $\beta_m,\dots,\beta_d$ satisfy 
\[2\beta_i=
\begin{cases}
t(\beta)-z_{m+1}/w_{m+1} &\text{if $i=m$,}\\
z_{i}/w_i-z_{i+1}/w_{i+1} &\text{if $m+1\leq i\leq d$.}
\end{cases}
\]
We conclude that the $d+1$ vectors $\uu_1,\dots, \uu_m,\vv_m,\dots, \vv_d$ generate all of $\RR^{d+1}$ and are therefore a basis.
$(2)$ The vector $\zz$ is a convex combination of a vector in the relative interior of the face $F_m$ and the vectors $\vv_m,\dots, \vv_d\in F^{\vee}$ if and only if $\alpha_1,\dots, \alpha_m>0$ and $\beta_m,\dots, \beta_d\geq 0$. 

Since $z_1/w_1\geq \dots \geq z_d/w_d\geq 0$, the formula for $\beta_i$ derived in the proof of part $(1)$ immediately implies that $\beta_{m+1},\dots,\beta_d\geq 0$. 

Moreover the formulas for $\beta_m$ and $t(\beta)$ in the proof of part $(1)$ show that $\beta_m\geq 0$ if and only if the following inequality holds 
\[ \frac{-az_0+\sum_{i=1}^m w_iz_i}{a^2+\sum_{i=1}^m w_i^2} \geq \frac{z_{m+1}}{w_{m+1}},\]
or equivalently, if $b_m\geq az_0$.

Finally, $\alpha_i$ is strictly positive for $i=1,\dots,m$ if and only if the ratios $\frac{\alpha_i}{w_i^2}=z_i/w_i-t(\beta)$ are strictly positive. The smallest of these ratios is achieved when $i=m$. Using the formula for $t(\beta)$ again, this quantity is positive if and only if the inequality
\[ \frac{z_m}{w_m}>\frac{-az_0+\sum_{i=1}^m w_iz_i}{a^2+\sum_{i=1}^m w_i^2} \]
holds, or equivalently, if $b_{m-1}<az_0$.

To finish the proof we need to verify that if $\zz$ is any point with $z_1/w_1\geq \dots\geq z_d/w_d\geq 0$ then there exists a unique index $m$ for which $b_{m-1}<az_0\leq b_{m}$. This follows from the fact that the $b_i$ form a non-decreasing sequence, as 
\[b_i-b_{i-1} =
\begin{cases}
\left(a^2+\sum_{r=1}^i w_r^2\right)\left(z_i/w_i-z_{i+1}/w_{i+1}\right) &\text{if $1\leq i\leq d-1$,}\\
\left(a^2+\sum_{r=1}^d w_r^2\right)z_d/w_d &\text{if $i=d$,}\\
\end{cases} \]  
is always nonnegative. It follows that for any $\zz\in \RR^{d+1}$ with $z_1/w_1\geq \dots\geq z_d/w_d\geq 0$ the projection in the direction of $C$ is given by $\pi_C(\zz)=\sum_{i=1}^m \alpha_i\uu_i=(s(\alpha)/a)\e_0+\sum_{i=1}^m(\alpha_i/w_i)\e_i$, as claimed. The formula for $\|\pi_C(\zz)\|^2$ follows by using the equalities $z_0+at(\beta)=s(\alpha)/a$
and $\alpha_i/w_i=z_i-w_it(\beta)$ derived in the proof of part $(1)$.
\end{proof}

Next we will relate the descent cones $D(I,\mathbf{w})$ of the norm $\|\cdot\|^w_1$ at a point $\mathbf{y}\in \RR^d$ with support $I := \supp(\mathbf{y})$ with the cones over weighted crosspolytopes from the previous lemma. As the descent cone depends only on $I$ and $\mathbf{w}$, in order to compute it we can assume that $\mathbf{y} = \frac{1}{k}\sum_{i \in I} \mathbf{e}_i/w_i$. 

\begin{lemma}\label{prop:isometric}
Suppose $|I|=k \neq d$. The descent cone $D(I,\mathbf{w})$ is isometric to the cone
 \[
  D'(I,\mathbf{w}) := \cone \{ \mathbf{e}_0/a \pm \mathbf{e}_i/w_i : i = 1, \dotsc, d-k\}
  \times \RR^{k-1},
 \]
 where $\mathbf{e}_0, \mathbf{e}_1, \dotsc, \mathbf{e}_{d-k}$ are the standard basis vectors of $\RR^{d-k+1}$, and 
 $a := \sqrt{\sum_{i \in I}w_i^2}$. 
\end{lemma}
\begin{proof} Let $B_{\mathbf w}$ be the unit ball in the weighted norm $\|\cdot\|_1^w$. Since $|\!|\mathbf{y}|\!|_1^{\mathbf w} = 1$ the equality 
$D(I,\mathbf{w}) = \cone \{\mathbf z - \mathbf y : \mathbf z \in B_{\mathbf w}\}$ holds. As a result
 \begin{align*}
  D(I,\mathbf{w}) &= \cone \{\pm \mathbf{e}_i/w_i - \mathbf{y} : i = 1, \dotsc, d\} \\
  &= \cone \{\pm \mathbf{e}_i/w_i - \mathbf{y} : i \notin I\} + \cone \{-\mathbf{e}_i/w_i - \mathbf{y} : i \in I\} \\
  & \hspace{85mm} + \cone \{\mathbf{e}_i/w_i - \mathbf{y} : i \in I\}.
 \end{align*}
 Since the generators of the last cone satisfy the relation 
 $\sum_{i \in I} (\mathbf{e}_i/w_i - \mathbf{y}) = 0$, 
we conclude that this cone equals the $(k-1)$-dimensional subspace
 \[
  L = \{ \mathbf{z} \in \RR^d : {\textstyle  \sum_{i\in I} w_i z_i = 0} 
  \text{ and } z_i=0 \text{ for any } i \notin I \}.
 \]
 We thus have
 \[
  D(I,\mathbf{w}) = \cone \{\pm \mathbf{e}_i/w_i - \mathbf{y} : i \notin I\} + 
  \cone \{-\mathbf{e}_i/w_i - \mathbf{y} : i \in I\} + L.
 \]
 Since $k \neq d$, the first summand in this expression contains the vector $-2\mathbf{y}$.
 It follows that the middle summand in the expression is redundant, because 
 for any $i \in I$ we have $(-\mathbf{e}_i/w_i-\mathbf{y}) + 2\mathbf{y} \in L$.
 Therefore
 \begin{equation}\label{eq:descent}
  D(I,\mathbf{w}) = \cone \{\pm \mathbf{e}_i/w_i - \mathbf{y} : i \notin I\} + L.
 \end{equation}
 
 Now, let $a := \sqrt{\sum_{i \in I}w_i^2}$. The vectors $\mathbf{e}_0', \mathbf{e}_1', \dotsc, \mathbf{e}_{d-k}'$ defined as 
 $\mathbf{e}'_0 := -\frac{1}{a}\sum_{i \in I} w_i \mathbf{e}_i$
 and $\{ \mathbf{e}'_1, \dotsc, \mathbf{e}'_{d-k} \} := \{ \mathbf{e}_i : i \notin I \}$ form 
 an orthonormal basis for the orthogonal complement $L^\perp$.
 We can write the generators of the cone in the right hand side of 
 Equation (\ref{eq:descent}) 
 as $\pm \mathbf{e}_i/w_i - y = (\mathbf{e}'_0/a \pm \mathbf{e}_i/w_i) - (y + \mathbf{e}'_0/a)$. The vector
 $\mathbf{y} + \mathbf{e}'_0/a$ is in the subspace $L$, so we have
 \begin{align*}
  D(I,\mathbf{w}) &= \cone \{(\mathbf{e}'_0/a \pm \mathbf{e}_i/w_i) - (\mathbf{y} + \mathbf{e}'_0/a) : i \notin I\} + L \\
  &= \cone \{ \mathbf{e}'_0/a \pm \mathbf{e}_i/w_i : i \notin I\} + L \\
  &= \cone \{ \mathbf{e}'_0/a \pm \mathbf{e}'_i/w_i : i = 1, \dotsc, d-k\} + L,
 \end{align*}
 from which the claimed result follows. 
\end{proof}

We are now ready to prove the main result of this section, which gives us explicit formulas for the projection onto arbitrary descent cones $D(I,\ww)$. For the reader's convenience we summarize the relevant notation below,

\noindent {\bf Notation.} For $\yy\in \RR^d$ let $I={\rm supp}(\yy)$ and $k:=|I|$. Define $J:=\{1,\dots, d\}\setminus I$ and $a:=\sqrt{\sum_{i\in I}w_i^2}$. Let $\qq_1,\dots, \qq_{k-1}$ be an orthonormal basis of the subspace $L=\{x\in \RR^d: \sum w_ix_i=0, \forall j\in J(x_j=0)\}$, and take $\e_0':=-\frac{1}{a}\sum_{i \in I} w_i \mathbf{e}_i$. Note that $\qq_1,\dots, \qq_{k-1}$, $\e_0'$ and the $\e_j$ with $j\in J$ form an orthonormal basis of $\RR^d$. We will describe the formula for the projection onto $D(I,\ww)$ in that basis.

\begin{theorem}\label{thm: main} For $\zz=z_0\e_0'+\sum_{j\in J}z_j\e_j +\sum_{j=1}^{k-1} g_k\qq_k\in \RR^{d}$ let $j_1,\dots, j_{d-k}$ be a permutation of $J$ such that $|z_{j_1}|/w_{j_1}\geq \dots \geq |z_{j_{d-k}}|/w_{j_{d-k}}$. Define 
\[b_l:=\begin{cases}
\sum_{i=1}^l w_{j_i}|z_{j_i}| -\frac{\left(a^2+\sum_{i=1}^lw_{j_i}^2\right)}{w_{j_l+1}}|z_{j_l+1}| &\text{if $1\leq l\leq d-k-1$,}\\
\sum_{i=1}^{d-k} w_{j_i}|z_{j_i}| &\text{if $l=d-k$.}
\end{cases}
\]
and let $m$ be the unique integer such that $b_{m-1}<az_0\leq b_m$ (with the convention that $b_{-1}=-\infty$ and $b_{d-k+1}=\infty$). If $t:=\frac{-az_0+\sum_{i=1}^m w_{j_i}|z_{j_i}|}{a^2+\sum_{i=1}^m w_{j_i}^2}$, $\alpha_i:=w_{j_i}|z_{j_i}|-w_{j_i}^2t$ and $\uu_i:=\e_0'/a+({\rm sign}(z_{j_i})/w_{j_i})\e_i$ for $i=1,\dots, m$, then the following statements hold:
\begin{enumerate}
\item The projection of $\zz$ towards $C:=D(I,\ww)$ is in the relative interior of a face of dimension $m+k-1$, and is given by the formula
\[\pi_C(\zz)=\sum_{i=1}^m \alpha_i\uu_i + \sum_{j=1}^{k-1} g_k\qq_k.\]

\item The following equality holds
\[\|\pi_C(\zz)\|^2=(z_0+at)^2+\sum_{i=1}^m (|z_{j_i}|-w_{j_i}t)^2+\sum_{j=1}^{k-1} g_j^2\]
\end{enumerate}
\end{theorem}
\begin{proof} By Lemma~\ref{prop:isometric} the projection is the direct sum of the identity in $L$ and the usual projection to the cone $D'(I,\ww)\subseteq L^{\perp}$. The correctness of the above formulas therefore follows from Lemma~\ref{lem: basic} part $(1)$ and Lemma~\ref{lem: key}.
\end{proof}

The previous Theorem suggests a simple Monte Carlo algorithm for estimating the expected statistical dimension $\overline{\delta}(\ww)$ for a given set of weights $\ww$, namely:

\begin{enumerate}
\item Generate $N$ independent samples $\mathbf{X}_i \sim \mathcal{F}$, $i=1,\dots, n$.
\item For each $i$ generate an independent gaussian random vector $\zz \in \RR^d$ and compute via Theorem~\ref{thm: main} the dimension $V_i$ of the face of $D({\rm supp}(\mathbf{X}_i),\ww)$ whose relative interior contains $\pi_C(\zz)$. 

\item Return $\bar{V}:=\frac{\sum V_i}{N}$.
\end{enumerate}

\begin{remark} By the orthonormality of the basis used in Theorem~\ref{thm: main} it is possible to sample a Gaussian random vector by putting independent standard normal coefficients in this basis. The dimension computation only depends on the coefficients $z_0$ and $z_j$ for $j\in J$, so only those need to be sampled reducing the computation time.
\end{remark}

How many samples $V_i$ are enough to get a good estimation of $\overline{\delta}(\ww)$? Proposition \ref{prop:hoeffding} answers this question. 

\begin{proposition}\label{prop:hoeffding} Let $(V_i)$ be i.i.d. random variables generated by the above Algorithm with input $\mathbf{w}$ and $\mathcal{F}$. If $\bar{V} = \frac{1}{n}\sum_{i=0}^n V_i$ then
$\PP(|\bar{V} - \bar{\delta}(\mathbf{w})| > t ) \leq \varepsilon$ whenever $n \geq \frac{\log(2/\varepsilon) d^2}{2t^2}.$
\end{proposition}
\begin{proof}
The variables $(V_i)_i^n$ have a bounded range, between $0$ and $d$. Therefore, the hypotheses for Hoeffding's inequality, \cite{hoeffding}, are fulfilled and 
\[\PP(|\bar{V} - \overline{\delta}(\ww)| > t ) \leq 2\exp\left(-\dfrac{2nt^2}{d^2}\right). \]
By taking $n$ as in the theorem the result follows. 
\end{proof}

\section{A Monte Carlo gradient descent algorithm}\label{section:DescentAlgo}

In this section we develop a Monte Carlo gradient descent algorithm for minimizing $\overline{\delta}(\mathbf{w})$. Our algorithm depends on having an analytic formula for the derivative of the squared length of the projection onto a descent cone. The relationship between this formula and the statistical dimension is explained by the following lemma. 

\begin{proposition} \label{prop:commutative} Let $\mathbf{X}_0$ be a random vector with distribution $\mathcal{F}$ and let $I = \text{supp}(\mathbf{X}_0)$. Take $C:= D(I,\mathbf{w})$ and $\bar{\delta}(\mathbf{w})$ the expected statistical dimension defined in \eqref{def:expectedStatisticalDim}. Then 
\[\dfrac{\partial \bar{\delta}(\mathbf{w})}{\partial w_s} = \EE_I \EE_\mathbf{g} \left(\dfrac{\partial \|\pi_{C}(\mathbf{g})\|^2}{\partial w_s}\right)\]
for any $s = 1, \dots, d$. 
\end{proposition}
\begin{proof}
It is possible to commute the differential operator $\frac{\partial(\cdot)}{\partial w_l}$ with the two expected values. Indeed, the first expected value  $\EE_I$ is a simple sum, and for $\EE_\mathbf{g}$ we use a measure theoretic version of the Leibniz integral rule.
\end{proof}

Motivated by the previous proposition we derive an analytic formula for the gradient of $\|\pi_{C}(\mathbf{g})\|^2$ with respect to the weights. The main difficulty lies in the fact that the numbers $g_k$ appearing in Theorem~\ref{thm: main} part $(2)$ depend on the weights $\ww$ for a fixed value of $\zz$. To capture this dependency we will compute the projection in the canonical basis, which uses the notation summarized in the paragraph preceding Theorem~\ref{thm: main}.

\begin{proposition} \label{prop:gradient} For $\zz=(z_1,\dots,z_d)\in \RR^d$ let $a:=\sqrt{\sum_{i\in I}w_i^2}$, $z_0:=-\sum_{i\in I} w_iz_i/a$, and let $j_1,\dots, j_{d-k}$ be a permutation of $J$ such that $|z_{j_1}|/w_{j_1}\geq \dots \geq |z_{j_{d-k}}|/w_{j_{d-k}}$. Define 
\[b_l:=\begin{cases}
\sum_{i=1}^l w_{j_i}|z_{j_i}| -\frac{\left(a^2+\sum_{i=1}^lw_{j_i}^2\right)}{w_{j_l+1}}|z_{j_l+1}| &\text{if $1\leq l\leq d-k-1$,}\\
\sum_{i=1}^{d-k} w_{j_i}|z_{j_i}| &\text{if $l=d-k$,}
\end{cases}
\]
and let $m$ be the unique integer such that $b_{m-1}<az_0\leq b_m$ (with the convention that $b_{-1}=-\infty$ and $b_{d-k+1}=\infty$). If the inequalities $|z_{j_1}|/w_{j_1}\geq \dots \geq |z_{j_{d-k}}|/w_{j_{d-k}}$ and $az_0\leq b_m$ are strict then the following formula holds:

\[\frac{\partial \|\pi_C(\zz)\|^2}{\partial w_s} = \frac{\partial}{\partial w_s}\left( (z_0+at)^2-z_0^2 + \sum_{i=1}^m (|z_{j_i}|-w_{j_i}t)^2\right) \]
where 
\[t=\frac{\sum_{i\in I} w_iz_i +\sum_{j=1}^m w_{j_i}|z_{j_i}|}{\sum_{j\in J}w_j^2 + \sum_{i=1}^m w_{j_i}^2}.\]
\end{proposition} 

\begin{proof} By Lemma~\ref{prop:isometric} the projection is the direct sum of the identity in $L$ and the usual projection to the cone $D'(I,\ww)\cap L^{\perp}$. As a result the equality $\pi_C(\zz)= \pi_{L}(\zz)+ \pi_C(\pi_{L^{\perp}}(\zz))$ holds. The right hand side is a decomposition into mutually orthogonal vectors and therefore the equalities
\[\|\pi_C(\zz)\|^2 = \|\pi_{L}(\zz)\|^2+\|\pi_C(\pi_{L^{\perp}}(\zz))\|^2 = \left(\|\zz\|^2-\|\pi_{L^{\perp}}(\zz)\|^2\right)+\|\pi_C(\pi_{L^{\perp}}(\zz))\|^2\]
hold. As shown in Lemma~\ref{prop:isometric} the vectors $\e_0'$ and $\{\e_j:j\in J\}$ form an orthonormal basis for $L^{\perp}$ and therefore 
\[\pi_{L^{\perp}}(\zz)=z_0\e_0'+\sum_{j\in J} z_j\e_j.\]
where $z_0$ is defined as above. Next we apply Theorem~\ref{thm: main} part $(2)$ to $\pi_{L^{\perp}}$ and conclude that  
\[\|\pi_C(\zz)\|^2= \|z\|^2-\left(z_0^2+\sum_{j\in J}z_j^2\right) + (z_0+at)^2+\sum_{i=1}^m (|z_{j_i}|-w_{j_i}t)^2.\]
where $t$ is given by the above expression.
Moreover this expression is valid in an open neighborhood of $\ww$ since all the above inequalities are assumed to be strict.
Since the $z_i$ for $i\neq 0$ do not depend on the weights $\ww$ we conclude that
\[\frac{\partial \|\pi_C(\zz)\|^2}{\partial w_s} = \frac{\partial}{\partial w_s}\left( (z_0+at)^2-z_0^2 + \sum_{i=1}^m (|z_{j_i}|-w_{j_i}t)^2\right) \]
as claimed.
\end{proof}

The previous two propositions suggest a Monte Carlo algorithm for estimating the gradient of the expected statistical dimension $\nabla_{\ww}\bar{\delta}(\ww)=\EE_I \EE_\mathbf{g} \left(\dfrac{\partial \|\pi_{C}(\mathbf{g})\|^2}{\partial w_s}\right)$, namely:

\begin{enumerate}
\item Generate $N$ independent samples $\mathbf{X}_i \sim \mathcal{F}$ in $\RR^d$, $i=1,\dots, n$.
\item For each $i$ generate an independent gaussian random vector $\zz \in \RR^d$ and compute via Proposition~\ref{prop:gradient} the vector $D_i:=\nabla_{\ww}\left(\|\pi_C(\zz)\|^2\right)$ where $C=D({\rm supp}(\mathbf{X}_i),\ww)$. 
\item Return the vector $\bar{D}:=\frac{\sum D_i}{N}$.
\end{enumerate}

Using this procedure we developed a method of steepest descent for approximating the weights $\mathbf w$ that minimize the expected statistical dimension $\bar{\delta}$. In each iteration, we compute an estimate $\bar{D}$ and we aim to walk in the direction opposite to it. To decide about the step size we use a backtracking-like approach, i.e. we start with a step size $\tau$ that does not violate the nonnegativity of the weights and then we use the Monte Carlo from the previous Section to check if $\mathbf{w}_k - \tau \bar{D}$ makes the expected statistical dimension smaller; if it is the case then we update the weights, otherwise we set $\tau = \tau/2$ and repeat. In the following section, we present some numerical examples on the practical performance of this algorithm. 

\section{Numerical examples}\label{Applications}

We consider three numerical examples. In two of them, first and third, we obtain promising results for the way we choose our weights: our recovery algorithm using the suitably weighted $\ell_1$-norm outperforms the non-weighted approach. To select the weights we employed two methods, the one described in Theorem \ref{teo:ProtoOptimal} and the numerical algorithm described in Section \ref{section:DescentAlgo}. In the second example we show a particular setting in which our approach might not always be better. 

We shall note that we always initialize our numerical algorithm to find the weights with the vector $\mathbf{w} = (1, \cdots, 1)$. Interestingly, if we initialize the algorithm with the weights of Theorem \ref{teo:ProtoOptimal}, the numerical method fails to find a non-negligible step size to continue. This suggests that these weights are a local optimum.

All the experiments were performed using MATLAB and the CVX package with Gurobi as solver. We present three kinds of figures: Recovery Frequency, Expected Intrinsic Volumes and Histograms of the statistical dimensions. To draw the Recovery Frequency figures we executed the following procedure: for each $m$, number of measurements, generate $100$ independent instances of each problem with the given distribution and with them estimate the frequency of perfect recovery. We defined $10^{-5}$ to be our success tolerance. For the Expected Intrinsic Volumes figures, we ran the next algorithm: generate $1000$ supports with the given distribution and for each support generate $100$ points following the algorithm in Section \ref{section: Formulas}, count the frequency to estimate $\bar{\nu}_k$ for all $k$. We only present the histogram in the first experiment, where we explain how we made it.


\subsection{Independent Bernoulli entries}\label{Bernoulli}
\begin{figure}[tbp]
\begin{subfigure}[b]{0.48\textwidth}
                \includegraphics[width=\textwidth]{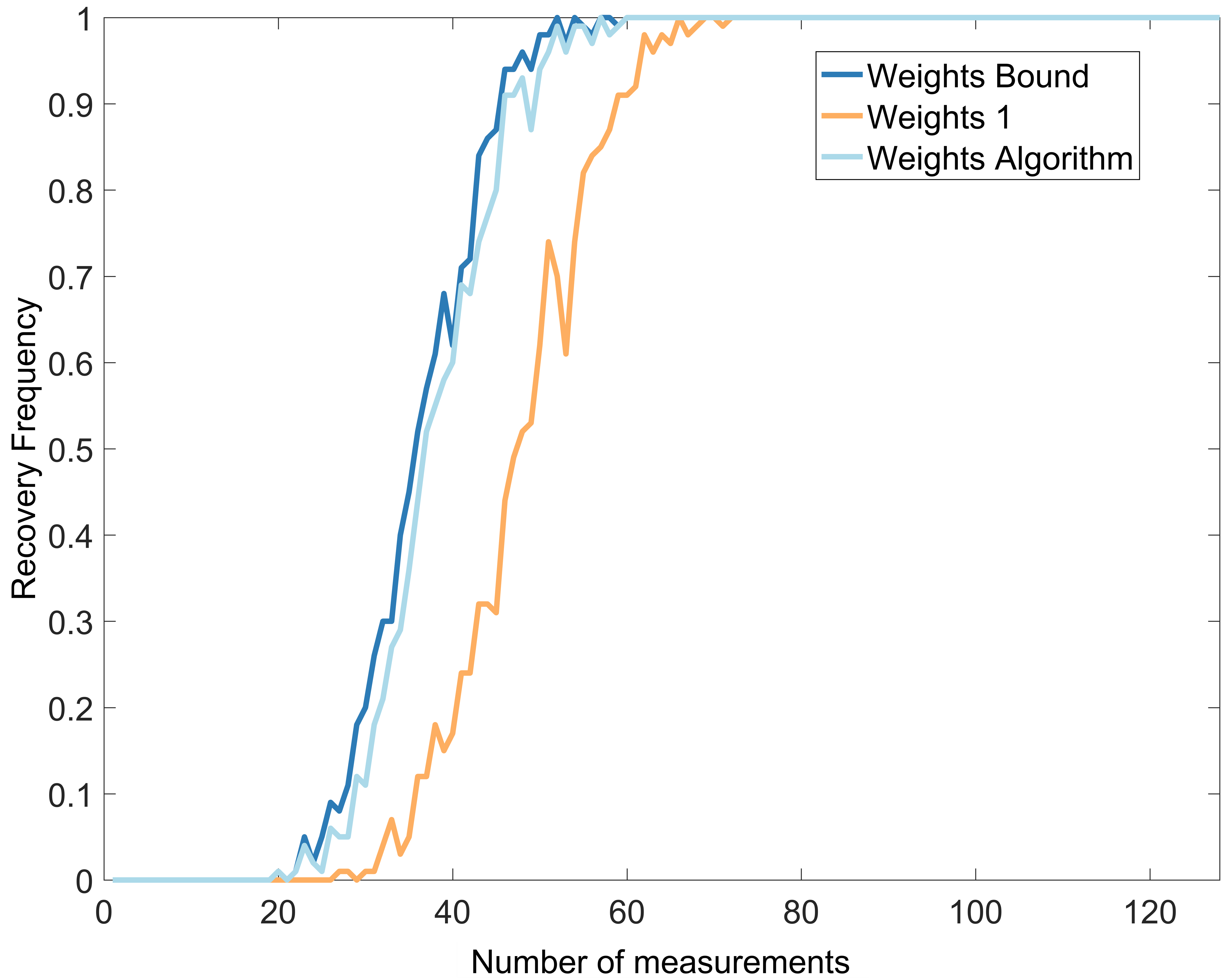}
        \end{subfigure}
        ~ 
        \begin{subfigure}[b]{0.48\textwidth}
                \includegraphics[width=\textwidth]{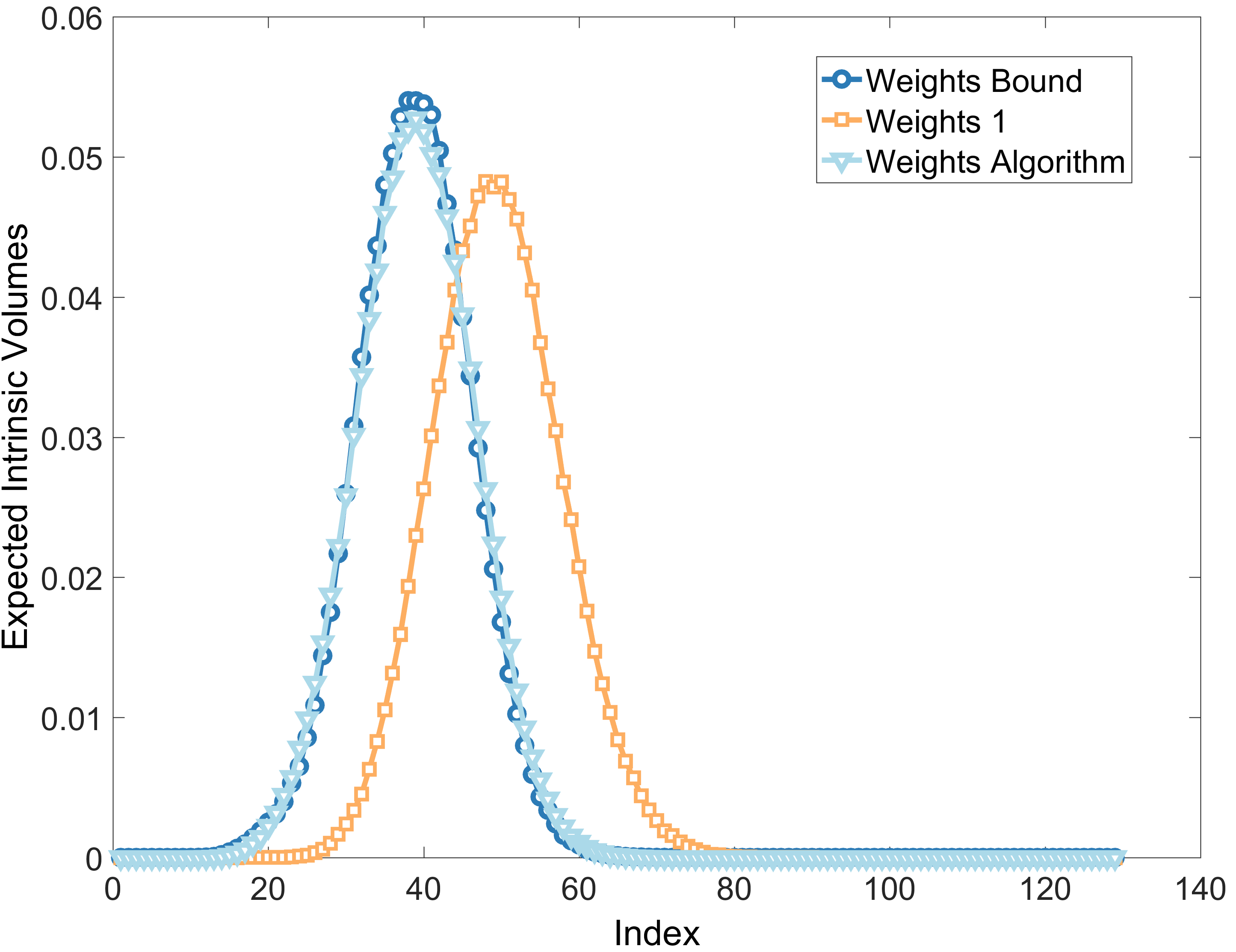}
        \end{subfigure}
  \centering
  \caption{The figure on the left shows the estimated recovery frequency; the figure on the right displays the expected intrinsic volumes estimated with our Monte Carlo. The dark blue lines are the results with our weights from Theorem \ref{teo:ProtoOptimal}, the orange lines are the results with the classical approach, and the light blue lines are the results with the weights obtained with the algorithm of Section \ref{section:DescentAlgo}.  }
  \label{fig:bernulli}
\end{figure}
For the first example we generate random vectors $\mathbf{X}_0 \in \RR^{128}$ using the following distribution: we partition the entries of $\mathbf{X}_0$ into 8 blocks of the same length, and in every block we take the entries to be i.i.d. random variables with a Bernoulli distribution, where the distribution parameter is defined by the index of the block. The parameters are given by
\vspace{10pt}
\begin{center}
\begin{tabular}{|c|c|c|c|c|c|c|c|}
\hline
 \hspace{40pt} & \hspace{40pt}  &\hspace{15pt}  & \hspace{15pt} & \hspace{15pt} & \hspace{15pt} & \hspace{15pt} & \hspace{40pt} 
 \\   \hline
\omit\mathstrut\upbracefill & \omit\mathstrut\upbracefill &\omit & \omit & \omit\dots   &\omit  & \omit& \omit\mathstrut\upbracefill \\
\omit $B(1, 2^{-1})$  &\omit $B(1, 2^{-2})$ & \omit & \omit & \omit & \omit& \omit& \omit $B(1, 2^{-8})$
\end{tabular}. 
\end{center}
\vspace{10pt}
Since the probability of being non-zero decreases exponentially in every block, then vectors with this distribution are sparse with high probability. 
For this particular example we present an histogram of the statistical dimensions of cones generated by the points with this distribuition. In each histogram, we draw $1000$ random supports $I$ and we estimate $\delta(I,\mathbf{w})$ with a Monte Carlo using $100$ points as described at the end of Section \ref{section: Formulas}. For this case, these statistical dimensions are very concentrated, as Figure \ref{fig:bernulliHist} shows. Thus, the hypothesis for Theorem \ref{theorem:concentratedDimensions}  are satisfied for as small parameter $\eta$. Therefore, problem \eqref{problem:w}, with our weights, is guaranteed to have higher success probability than \eqref{problem:1}.

\begin{figure}[htbp]
\begin{subfigure}[b]{0.48\textwidth}
                \includegraphics[width=\textwidth]{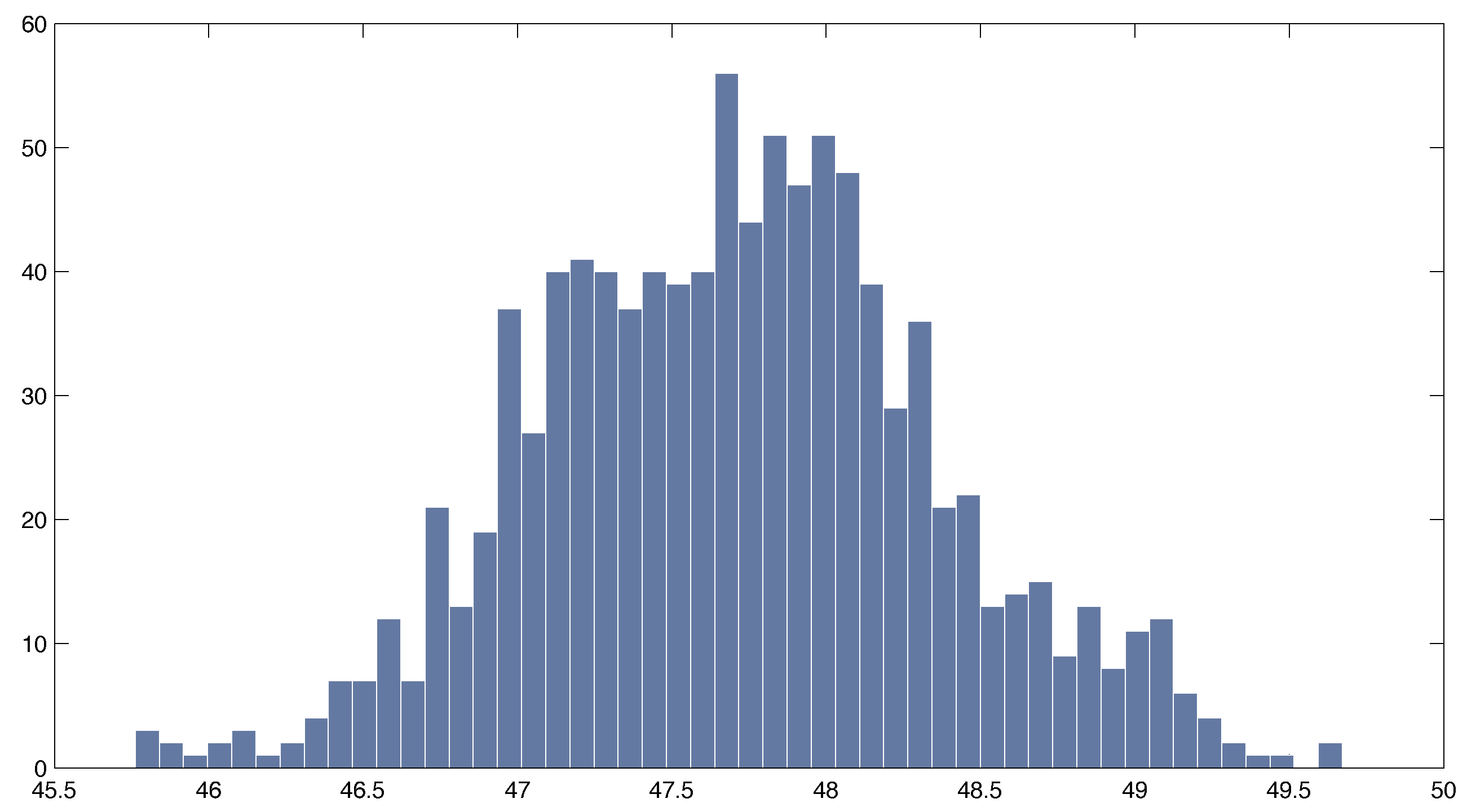}
        \end{subfigure}
        ~ 
        \begin{subfigure}[b]{0.48\textwidth}
                \includegraphics[width=\textwidth]{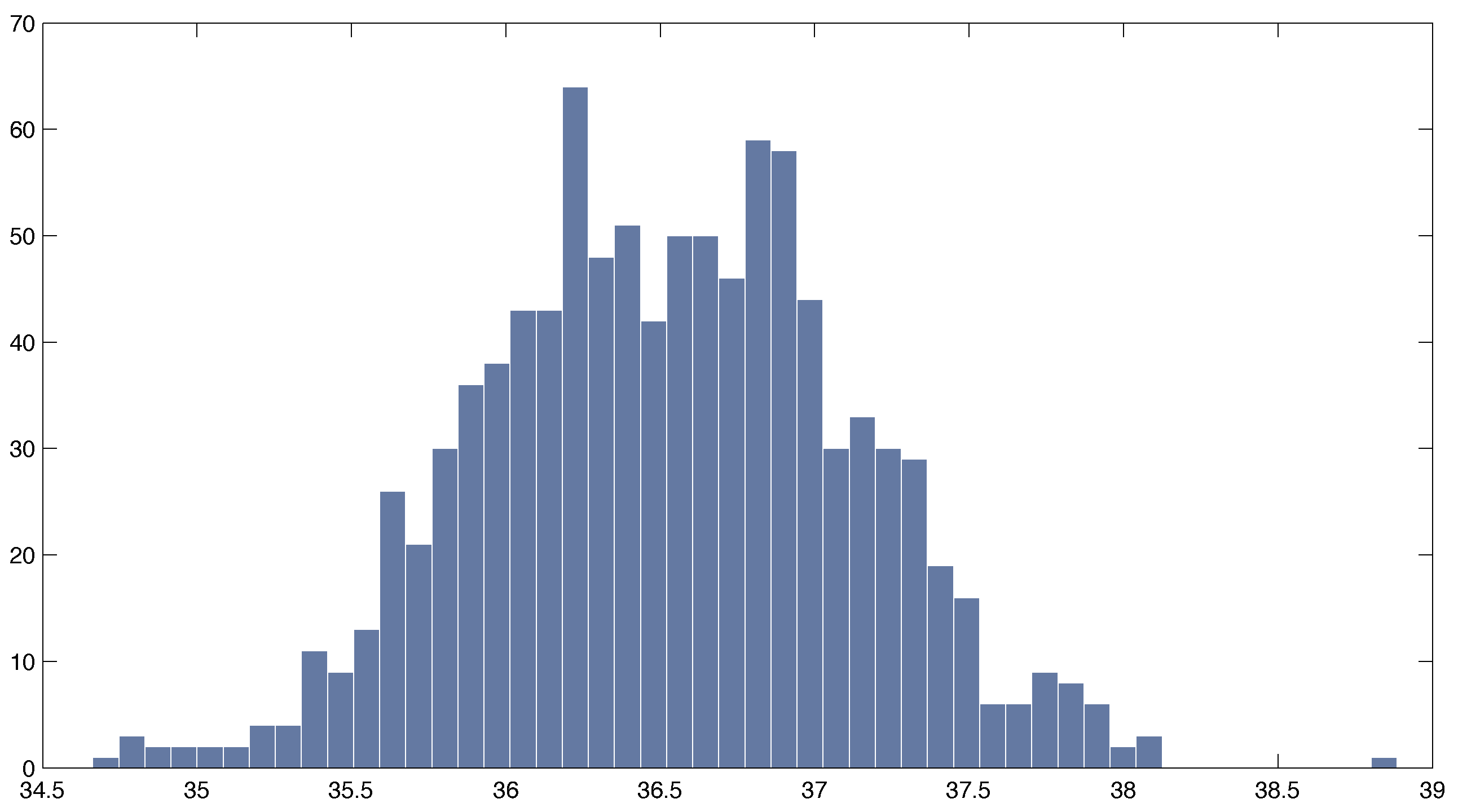}
        \end{subfigure}
  \centering
  \caption{Histograms of the statistical dimensions of random cones generated with independent Bernoulli entries. On the left the histogram with cones with weights one. On the right the histogram with cones with the weights found with \eqref{eq:formulaWeights}. }
  \label{fig:bernulliHist}
\end{figure}

\subsection{A non-sharp case}\label{NonSharp} For this experiment our choice of weights is not always the best. Here we consider an artificial distribution with four possible supports, each one with probability $1/4$, as in Figure \ref{fig:badCaseWeights}.
\begin{figure}[htbp]
  \centering
  \includegraphics[width=0.95\textwidth]{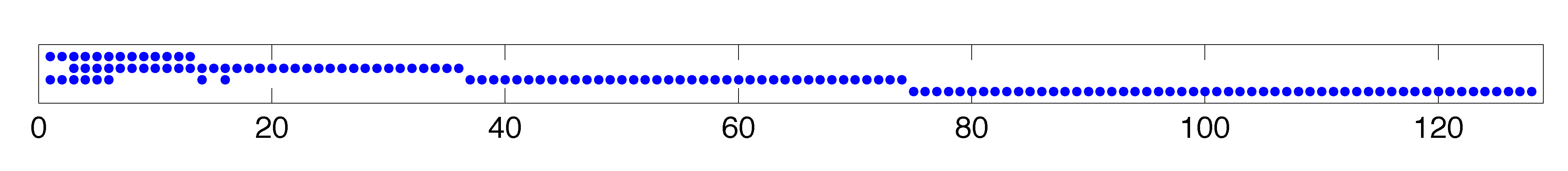}
  \caption{Possible supports: Each row represents a support; a blue point is a $1$ and white is a $0$.}
  \label{fig:badCaseWeights}
\end{figure}
One particular characteristic of this distribution is that the intrinsic volumes for the descent cones corresponding to different supports are concentrated around different locations. Since all the supports have equal probability, the expected intrinsic volumes are not concentrated, as we show in Figure \ref{fig:badCaseNus}.  Intuitively, what Theorem \ref{theorem:concentratedDimensions} is showing in this case is that the transition is not sharp, and therefore by minimizing $\bar{\delta}$ we are not necessarily reducing the probability of failure, i.e., the tail of the intrinsic volumes. We ran the descent algorithm proposed in Section \ref{section:DescentAlgo} starting at weights one. After two iterations, we obtained very similar weights to the ones found using the bound in Theorem \ref{teo:ProtoOptimal}. 
\begin{figure}[htbp]
  \centering
  \begin{subfigure}[b]{0.48\textwidth}
                \includegraphics[width=\textwidth]{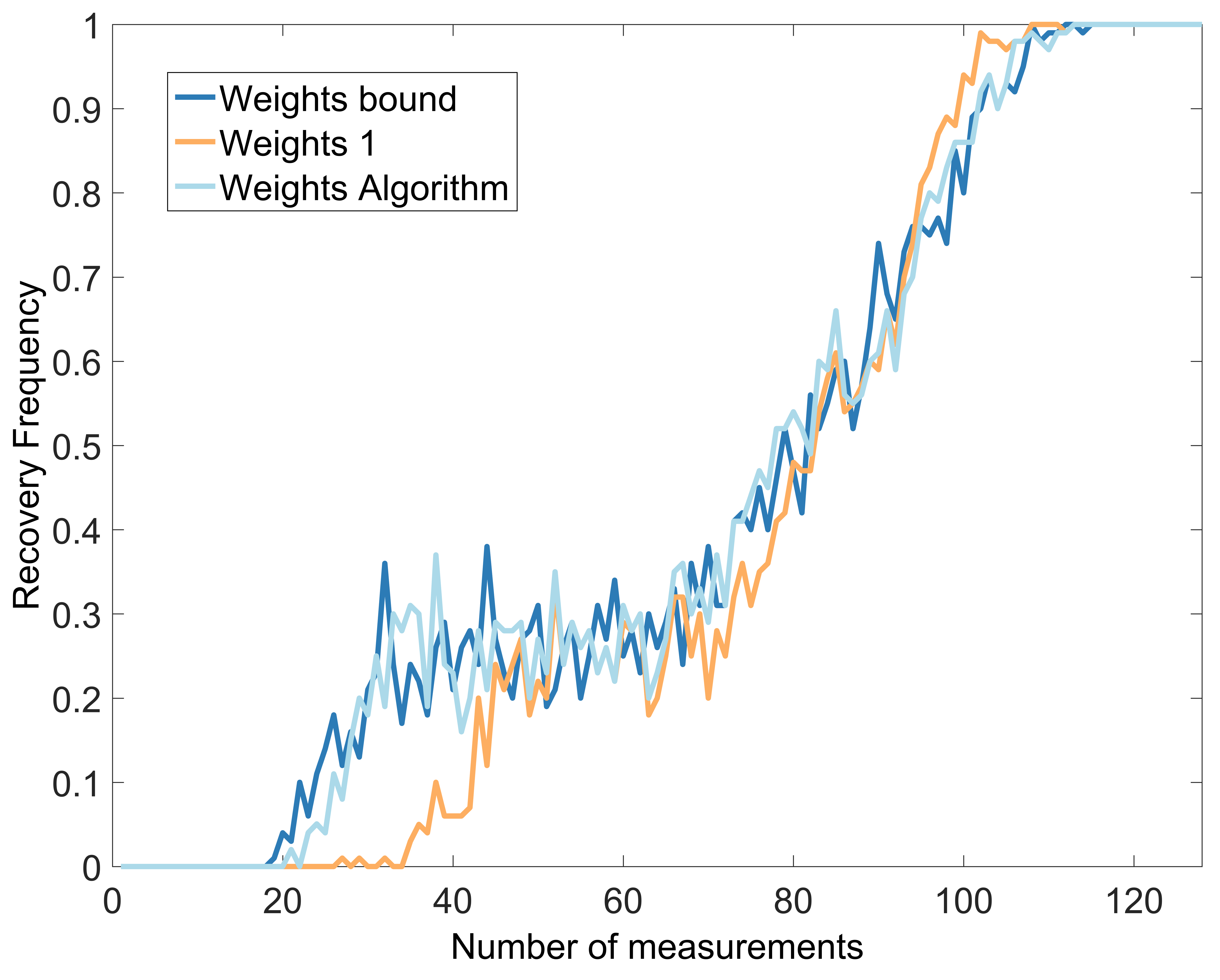}
        \end{subfigure}
        ~ 
        \begin{subfigure}[b]{0.48\textwidth}
                \includegraphics[width=\textwidth]{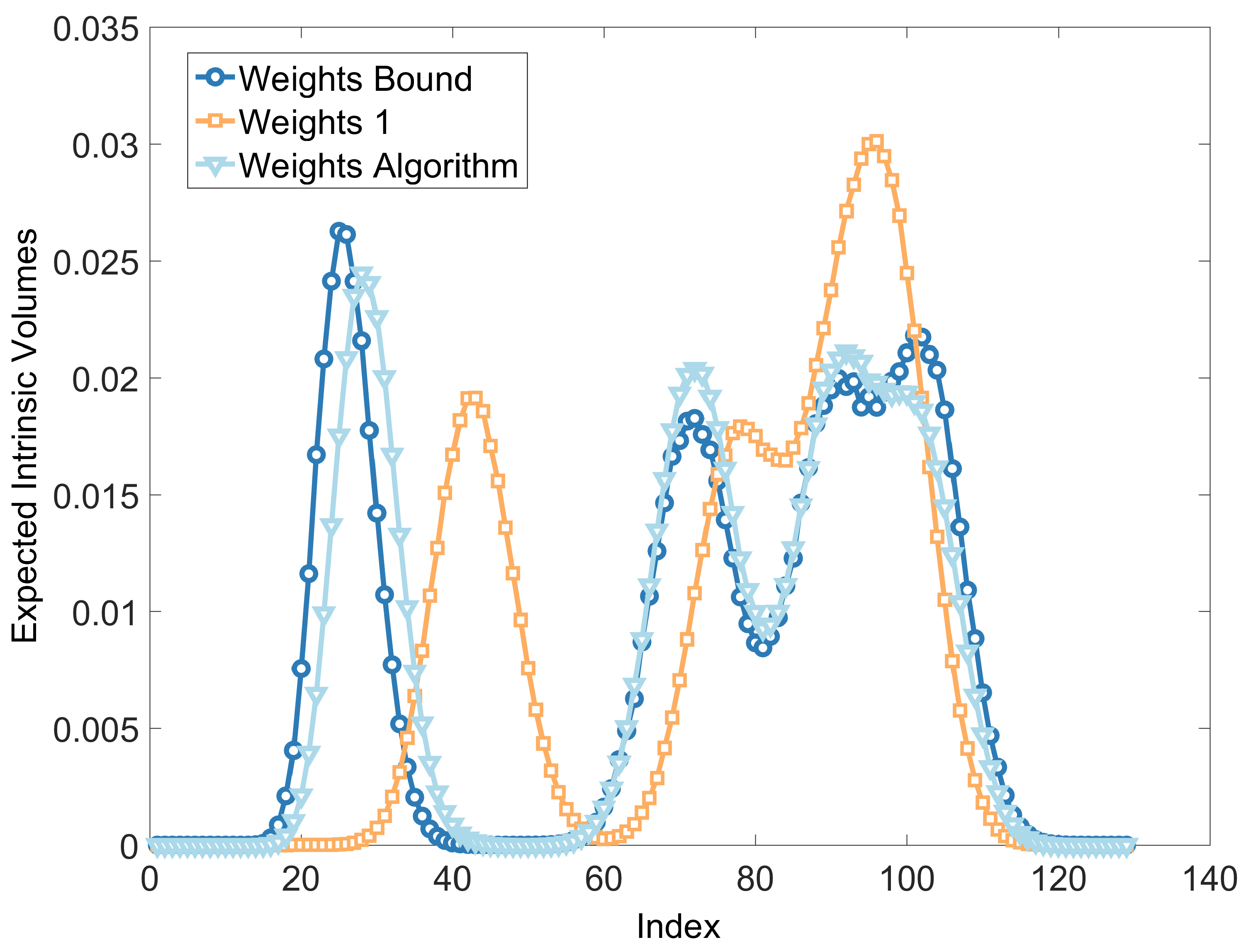}
        \end{subfigure}
  \caption{Expected instrinsic volumes of the non-sharp experiment. Same conventions as in Figure \ref{fig:bernulli}.}
  \label{fig:badCaseNus}
\end{figure}

\subsection{MRI}
\begin{wrapfigure}{r}{0.32\textwidth}
  \begin{center}
    \includegraphics[width=0.3\textwidth]{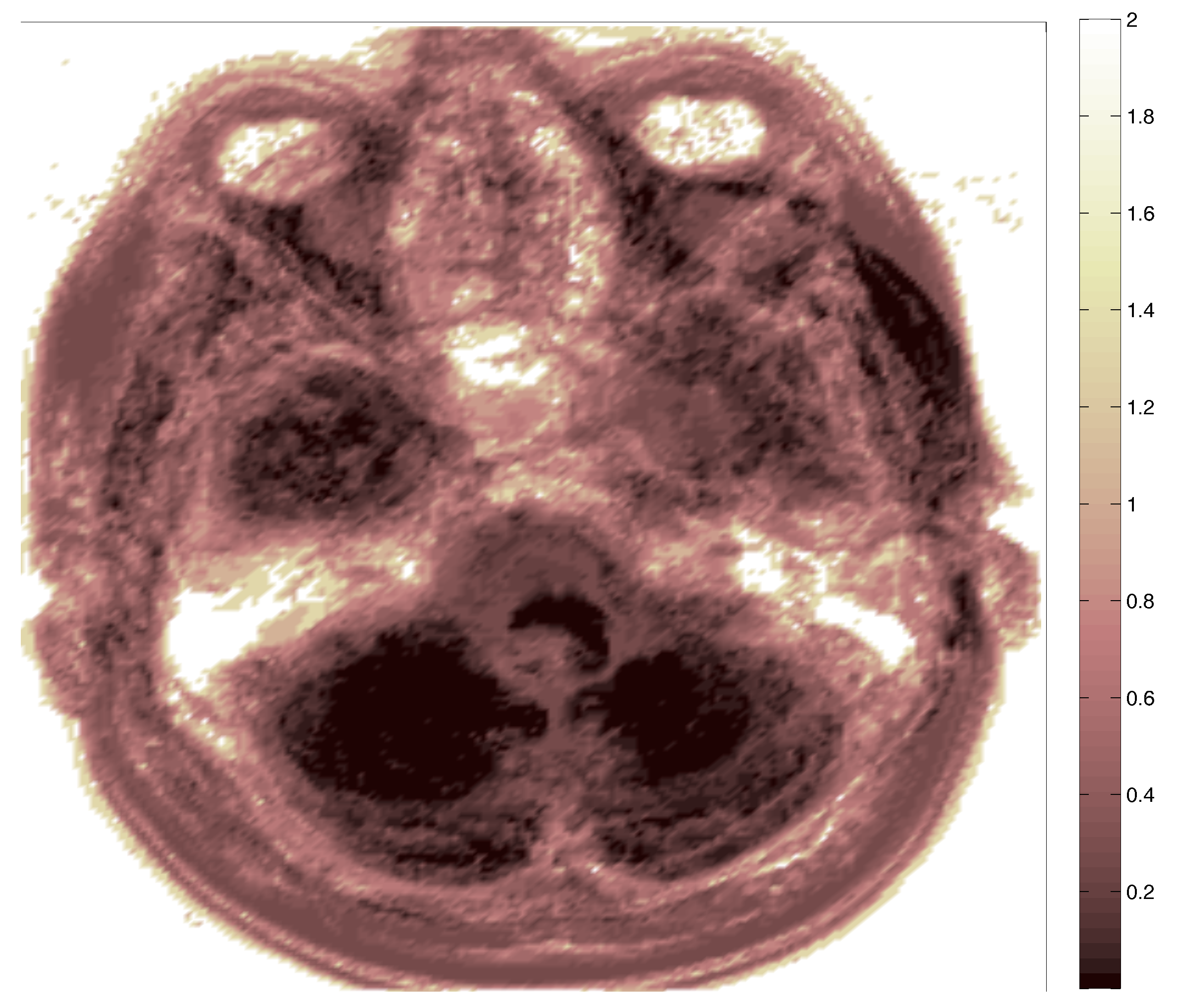}
  \end{center}
  \caption{Heat map of one of the weights found.}
\end{wrapfigure}
We took real brain MRI from 5 patients. The resonances were composed of multiple 2D slides of the brain. In order to promote a sensible distribution in this setting, we restricted only to the slides at eye-level height. Subsequently, we centered and cut the images, increasing the ratio between the non-zero entries and the size of the image as much as possible. After this process we ended up with $47$ grayscale images of size $215 \times 184$ pixels. 

We performed a row-by-row reconstruction and tested the weights described in Theorem \ref{teo:ProtoOptimal}, using a leave-one-out cross-validation to measure the frequency of perfect recovery for several number of measurements $m$. In other words, we selected the $i$th image and used the rows of the other $46$ images to obtain the empirical distribution $\hat{\mathcal{F}_i}$, which we then used to compute the weights $\hat{\mathbf{w}}_i$ and measure the frequency of perfect recovery of the $i$th image. We repeated the procedure for all the images and took the average of the frequencies. Figure \ref{fig:MRI} shows the results.
\begin{figure}[htbp]
  \centering
  \includegraphics[width=0.65\textwidth]{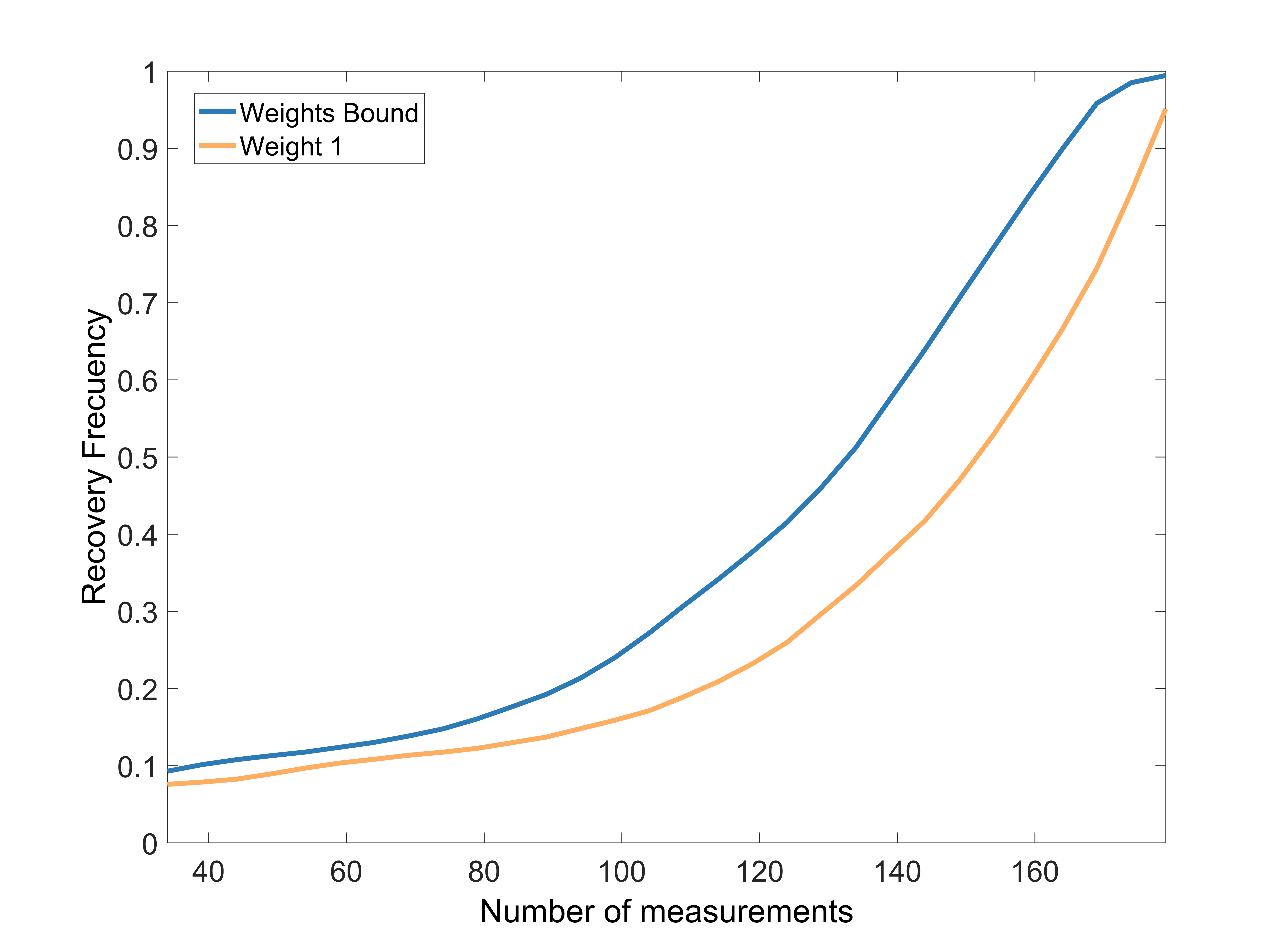}
  \caption{Frequency of perfect recovery of the MRI experiment. Same conventions as in Figure \ref{fig:bernulli}. }
  \label{fig:MRI}
\end{figure}

\section{Conclusions and open questions} \label{section:conclusions}
In this work, we showed that it is possible to take advantage of prior statistical information from a signal, i.e. its support distribution, to improve the standard compressed sensing method. In particular, we developed a method to shift the inflection point of the statistical dimension by minimizing an appropriately weighted $\ell_1$-norm. To do so, we presented two ways to find good weights. Our methods pick the weights aiming to minimize the expected statistical dimension, $\bar{\delta}(\ww)$. The first method uses an explicit formula, \eqref{eq:formulaWeights}, based on an upper bound and the other uses a numerical algorithm based on a Monte Carlo gradient descent procedure. 

Moreover we proved through experiments that the proposed methods are effective in many contexts in the sense that they increase the success probability for any $m$. However, in the case where the expected intrinsic volumes were not concentrated around their mean, our methods do not always beat the unweighted approach. It appears that under these circumstances minimizing the statistical dimensions may spread the expected intrinsic volumes, increasing their variance and their tail, and thus the failure probability. It is therefore natural to ask: What to minimize when the statistical dimension does not work? and how to find optimal weights in those cases? We believe that one way to solve this problem is by fixing $m$, the number of measurements, and choosing weights that minimize the probability of failure.

\section*{Acknowledgments}
The first, second and fourth authors were supported by Universidad de los Andes under the Grant ``Fondo de Apoyo a Profesores Asistentes''(FAPA). We would like to thank Mario Andres Valderrama for providing us with the brain MRI data. We would also like to thank Dennis Amelunxen, Martin Lotz and Javier Pe\~na for useful conversations during the completion of this work. Finally, we would like to thank the anonymous reviewers for their thorough reading of this manuscript and
for their recommendations.


\bibliographystyle{alpha}
  \bibliography{biblio}

\end{document}